\documentclass[a4paper,reqno,final]{amsart}

\usepackage{amssymb}
\usepackage{mathtools}
\usepackage{booktabs}
\usepackage[margin=10pt]{caption}
\usepackage{dsfont}

\newcommand{\set}[1]{\left\{#1\right\}}
\newcommand{\abs}[1]{\left\lvert#1\right\rvert}
\newcommand{\integer}{\mathds Z}
\newcommand{\nat}{\mathds N}
\newcommand{\real}{\mathds R}
\newcommand{\Normal}{\mathcal N}
\DeclareMathOperator{\Ee}{\mathds E}
\DeclareMathOperator{\sgn}{\mathrm sgn}

\newtheorem{theorem}{Theorem}[section]
\newtheorem{proposition}[theorem]{Proposition}
\newtheorem{lemma}[theorem]{Lemma}
\newtheorem{corollary}[theorem]{Corollary}

\newtheorem{conjecture}{Conjecture}

\theoremstyle{remark}
\newtheorem{remark}[theorem]{Remark}

\usepackage[german=guillemets]{csquotes} 
\setquotestyle[british]{english}

\usepackage[UKenglish]{babel} 

\makeatletter
\@namedef{subjclassname@2020}{%
  \textup{2020} Mathematics Subject Classification}
\makeatother

\renewcommand\labelenumi{\textup{\arabic{enumi}.}}
\renewcommand\theenumi\labelenumi
\begin{document}

\title[Fractional Gaussian noise: analytical \& computational problems]{\bfseries Analytical and computational problems related to fractional Gaussian noise}

\author[Yu.\ Mishura]{Yuliya Mishura}
\address{Taras Shevchenko National University of Kyiv, Department of Probability Theory, Statistics and Actuarial Mathematics, 64/13, Volodymyrska, 01601 Kyiv, Ukraine}
\email{yuliyamishura@knu.ua}

\author[K.\ Ralchenko]{Kostiantyn Ralchenko}
\address{Taras Shevchenko National University of Kyiv, Department of Probability Theory, Statistics and Actuarial Mathematics, 64/13, Volodymyrska, 01601 Kyiv, Ukraine}
\email{kostiantynralchenko@knu.ua}

\author[R.L.~Schilling]{Ren\'e L.\ Schilling}
\address{TU Dresden\\ Fakult\"{a}t Mathematik\\ Institut f\"{u}r Mathematische Stochastik\\ 01062 Dresden, Germany}
\email{rene.schilling@tu-dresden.de}


\begin{abstract}
    We study the projection of an element of fractional Gaussian noise onto its neighbouring elements. We prove some analytic results for the coefficients of this projection, in particular, we obtain recurrence relations for them. We also make several conjectures concerning the behaviour of these coefficients, provide numerical evidence supporting these conjectures, and study them theoretically in particular cases. As an auxiliary result of independent interest, we investigate the covariance function of fractional Gaussian noise, prove that it is completely monotone for $H>1/2$, and, in particular monotone, convex, log-convex along with further useful properties.
\end{abstract}

\maketitle

\section{Introduction}\label{section1}

This paper is about some (conjectured) properties of the projection  of an element of fractional Gaussian noise onto the neighbouring elements. Unfortunately, not all our conjectures are amenable to analytical proofs, while numerical experiments confirm their validity. This is indeed rather strange, as the properties of fractional Brownian motion and its increments have been thoroughly studied, attracting a lot of research efforts resulting in countless papers and several books, e.g.\ \cite{biagini,Mishura2008,nourdin,samorodnitsky}. These books are mostly devoted to the stochastic analysis of fractional processes, the properties of their trajectories, distributional properties of certain functionals of the paths, and related issues. There is, however, an area where much less is known: problems relating to the covariance matrix of fractional  Brownian motion and fractional Gaussian noise in high dimensions, and its determinant. Computational features of the covariance matrices are widely used for simulations and in various applications, see e.g.\ \cite{dieker,gupta,kijima,montillet}. The problem considered in the present paper arose in the following way: In \cite{risks} the authors construct a discrete process that converges weakly to a fractional Brownian motion (fBm) $B^H=\{B^H_t, t\ge 0\}$ with Hurst parameter $H\in(\frac12,1)$. The construction of this process is based on the Cholesky decomposition of the covariance matrix of the fractional Gaussian noise (fGn). Several interesting properties of this decomposition are proved in \cite{risks} such as the positivity of all elements of the corresponding triangular matrix and the monotonicity along its main diagonal. Numerical examples suggest also the conjecture, that one has monotonicity along all diagonals of this matrix. However, the analytic proof of this fact remains an open problem. Studying this problem, the authors of the paper \cite{risks} establish a connection between the predictor's coefficients -- that is, the coefficients of the projection of any value of a stationary Gaussian process onto finitely many subsequent elements -- and the Cholesky decomposition of the covariance matrix of the process. It turns out that the positivity of
the coefficients of the predictor implies the monotonicity along the diagonals of the triangular matrix of the Cholesky decomposition of fGn, which is sufficient for the monotonicity along the columns of the triangular matrix in the Cholesky decomposition of fBm itself; this property, in turn, ensures the convergence of a wide class of  discrete-time schemes to a fractional Brownian motion. We will see in Section~\ref{ssec:system} below, that the coefficients of the predictor can be found as the solution to a system of linear equations, whose coefficient matrix coincides with the covariance matrix of fGn. This enables us to reduce the monotonicity problem for the Cholesky decomposition to proving the positivity of the solution to a linear system of equations. However, see Section~\ref{section2}, even in the particular case of a $3\times3$-matrix, an analytic proof of positivity of all coefficients is a non-trivial problem. For the moment, we have only a partial solution. Therefore, we formulate the following conjecture:

\begin{conjecture}\label{con:A}
    If $H>1/2$, then the coefficients of the projection of any element of fractional Gaussian noise onto any finite number of its subsequent elements are strictly positive.
\end{conjecture}

Due to stationarity, it is sufficient to establish \ref{con:A} for
\begin{gather*}
    \Ee \left(\Delta_1 \mid \Delta_2,\dots,\Delta_n\right),\quad n\ge 3,
\end{gather*}
where $B^H$ denotes fBM and $\Delta_k=B^H_k-B^H_{k-1}$.  Having computational evidence but lacking an analytical proof for conjecture \ref{con:A}, we provide in this paper a wide range of associated properties of coefficients, some with an analytic proof, and some obtained using various computational tools. It is, in particular, interesting to study the asymptotic behaviour of the coefficients as $H\uparrow 1$. This is particularly interesting since $H=1$ fractional Brownian motion $B^1$ is degenerate, i.e.\ $B^1_t= t\xi$, where $\xi\sim\Normal(0,1)$, and $\Normal(0,1)$ denotes the standard normal distribution. Consequently, $\Delta_k\sim\Normal(0,1)$ for all $k\ge 1$, and
\begin{gather*}
    \Ee \left(\Delta_1 \mid \Delta_2,\dots,\Delta_n\right)
    = \sum_{k=2}^n \alpha_k \Delta_k
    \sim \Normal(0,1)
\end{gather*}
for any convex combination $\alpha_k\ge 0,\; \sum_{k=2}^n \alpha_k=1$. This shows that in the case $H=1$ the values of the coefficients are indefinite, and therefore they cannot define the  asymptotic behaviour of the prelimit coefficients as $H\uparrow 1$. It will be very \enquote{elegant} if  all coefficients tend to $(n-1)^{-1}$, however, in reality  their asymptotic behaviour is different, see Section~\ref{ssec:partic}. Another interesting question are the relations between the coefficients. It is natural to assume that they decrease as $k$ increases, but the situation here is also more involved, essentially depending on the value of $H$.
In Section~\ref{ssec:recurrent} we prove some recurrence relations between the coefficients. These relations lead to a computational algorithm which is more efficient than solving the system of equations as described in Section~\ref{ssec:system}. Finally, it turns out that the positivity of the first coefficient can be proven analytically for all values of $n$; this result is established in Section~\ref{ssec:first-pos}.

We close the paper with a few numerical examples, supporting our theoretical results and conjectures. In particular, we compute the coefficients for all $n\le 10$ and for various values of $H$, and discuss their behaviour. Also, we compare different calculation methods for the coefficients in terms of computing time, and we demonstrate the advantage of the approach via recurrence formulae in most cases.

The paper is organized as follows: Section \ref{section2} contains almost all properties of the predictor's coefficients that can be established analytically, and it introduces the system of linear equations for these coefficients and some properties of the coefficients of this system. We consider in detail two particular cases: $n=3$ and $n=4$. In these cases we  prove the positivity of all coefficients, establish some relations between them, and study the asymptotic behaviour as $H\uparrow1$. We also obtain recurrence relations for the coefficients, and prove that for all values of $n$ the first coefficient is positive. Section \ref{section3} contains some numerical illustrations of the properties and conjectures from Sections \ref{section1} and \ref{section2}. In Section~\ref{ssec:H<1/2} we briefly discuss some observations concerning the case $H<1/2$.

\section{Analytical properties of the coefficients}\label{section2}
Let $B^H = \set{B^H_t, t\ge0}$ be a fractional Brownian motion (fBm) with Hurst index $H\in(\frac12,1)$. We use
\begin{gather*}
    \Delta_n = B^H_n - B^H_{n-1},
    \quad n=1,2,3,\dots.
\end{gather*}
for the $n$th increment of fBM. It is well known that the process $B^H$ has stationary increments, which implies that $\set{\Delta_n,n\ge1}$ is a stationary Gaussian sequence (known as \emph{fractional Gaussian noise} -- fGn for short) with the following autocovariance function:
\begin{gather}\label{eq:rho_k}
    \rho_k = \Ee \Delta_1 \Delta_{k+1}
    =\frac12 \left(\abs{k+1}^{2H} - 2\abs{k}^{2H} + \abs{k-1}^{2H}\right),
    \quad k\ge 1.
\end{gather}
Obviously, $\rho_0 =1$. Since the joint distribution of $(\Delta_1, \dots, \Delta_n)$ is centred and Gaussian, we obtain the following relation from the theorem on normal correlation (see, e.g.\ \cite[Theorem 3.1]{tsp}):
\begin{gather}\label{eq:proj1}
    \Ee \left(\Delta_1 \mid \Delta_2,\dots,\Delta_n\right)
    = \sum_{k=2}^n \Gamma_n^k \Delta_k,
    \quad n\ge2,
\end{gather}
where $\Gamma_n^k\in \real$. Let us consider two approaches to calculation of the coefficients $\Gamma_n^k$. The first method is straightforward, it involves solving of the system of linear equations. The second one is based on recurrence relations for the $\Gamma_n^k$.

\subsection{System of linear equations for coefficients}\label{ssec:system}

Multiplying both sides of \eqref{eq:proj1} by $\Delta_l$, $2\le l\le n$, and taking expectations yields
\begin{gather*}
    \Ee \Delta_1 \Delta_l = \sum_{k=2}^n \Gamma_n^k \,\Ee \Delta_k \Delta_l,
    \quad 2\le l\le n.
\end{gather*}
Due to stationarity,
\begin{gather*}
    \Ee\Delta_k \Delta_l = \rho_{|l-k|}.
\end{gather*}
This leads to the following system of linear equations for the coefficients $\Gamma_n^k$, $k=2,\dots,n$:
\begin{gather}\label{eq:system}
    \rho_{l-1} = \sum_{k=2}^n \Gamma_n^k \rho_{|l-k|},
    \quad 2\le l\le n.
\end{gather}
We can solve this using Cramer's Rule,
\begin{gather*}
    \Gamma_n^k = \frac{\det A_k}{\det A},
\end{gather*}
where
\begin{align}\label{matrices}
A &=
\begin{pmatrix}
1          & \rho_1     & \rho_2     & \dots   & \rho_{k}    & \dots  & \rho_{n-2} \\
\rho_1     & 1          & \rho_1     & \dots   & \rho_{k-1}  & \dots  & \rho_{n-3} \\
\vdots     & \vdots     & \vdots     &         & \vdots      &        & \vdots     \\
\rho_{n-2} & \rho_{n-3} & \rho_{n-4} & \dots   & \rho_{n-k-1}& \dots  & 1
\end{pmatrix}
\intertext{and $A_k$ is the matrix $A$ with its $k$th column vector replaced by $(\rho_1, \dots, \rho_{n-1})^\top$:}\notag\\[-8pt]
\label{matrices-2}
A_k &=
\begin{pmatrix}
1          & \rho_1     & \rho_2     & \dots   & \smash{\overbracket[.6pt]{\;\rho_{1}\;}^{\mathclap{\text{$k$th column}}}}  & \dots  & \rho_{n-2} \\
\rho_1     & 1          & \rho_1     & \dots   & \rho_{2}  & \dots  & \rho_{n-3} \\
\vdots     & \vdots     & \vdots     &         & \vdots    &        & \vdots     \\
\rho_{n-2} & \rho_{n-3} & \rho_{n-4} & \dots   & \rho_{n-1}& \dots  & 1
\end{pmatrix}.
\end{align}

\begin{remark}
It is known that the finite-dimensional distributions of $B^H$ have a nonsingular covariance matrix; in particular, for any $0<t_1<\dots<t_n$, the values $B^H_{t_1},\dots,B^H_{t_n}$ are linearly independent, see \cite[Theorem~1.1]{book-approx} and its proof. Obviously, a similar statement holds for fractional Gaussian noise, since the vector $(\Delta_1,\dots,\Delta_n)$ is a nonsingular linear transform of $(B^H_1,\dots, B^H_n)$. In other words, $\det A\ne0$; moreover, if $\sum_k\alpha_k\Delta_k=0$ a.s., then $\alpha_k=0$ for all $k$.
\end{remark}

\subsection{Relations between the values $\rho_k$}

In order to establish analytic properties of the coefficients $\Gamma_n^k$ we need several auxiliary results on the properties of the sequence $\set{\rho_k,\,k\in\integer_+}$. We start with a useful relation between $\rho_1$, $\rho_2$  and $\rho_3$.
\begin{lemma}\label{l:r1r2r3}
The following equality holds:
\begin{gather}\label{eq:rho2viarho3}
    \rho_2-\rho_1^2
    = \frac12(\rho_1-\rho_3).
\end{gather}
\end{lemma}

\begin{proof}
Using the self-similarity of fBm and the stationarity of its increments, we get
\begin{align*}
    2^{2H}\rho_1
    &= 2^{2H}\Ee B^H_1(B^H_2-B^H_1) \\
    &= \Ee B^H_2(B^H_4-B^H_2) \\
    &= \Ee (B^H_2-B^H_1)(B^H_4-B^H_2) + \Ee B^H_1(B^H_4-B^H_2)\\
    &= \Ee\Delta_2(\Delta_3+\Delta_4) + \Ee\Delta_1(\Delta_3+\Delta_4)\\
    &= \rho_1 + 2\rho_2 + \rho_3.
\end{align*}
Note that by \eqref{eq:rho_k}, $\rho_1=2^{2H-1}-1$, whence $2^{2H} = 2 \rho_1+2$. Thus we arrive at
\begin{gather*}
    (2\rho_1 + 2)\rho_1 = \rho_1 + 2\rho_2 + \rho_3.
\end{gather*}
which is equivalent to \eqref{eq:rho2viarho3}.
\end{proof}

\begin{remark}
The inequality $\rho_1^2<\rho_2$ was proved in \cite[p.~28]{risks} by analytic methods. In this paper we improve this result in two directions: we obtain an explicit expression for $\rho_2-\rho_1^2$ and we prove the sharper bound $\rho_1^2<\rho_3$, see Lemma \ref{l:r12r3} below.
\end{remark}

Many important properties of the covariance function of a fractional Gaussian noise (such as monotonicity, convexity and log-convexity) follow from the more general property of complete monotonicity, which is stated in the next lemma. To formulate it, let us introduce the function
\begin{gather*}
    \rho(x)
    = \rho(H, x)
    = \frac12 \left(|x+1|^{2H} - 2|x|^{2H} + |x-1|^{2H}\right),\quad x\ge 0.
\end{gather*}

\begin{lemma}\label{l:cm}
\leavevmode
\begin{enumerate}
\item\label{l:cm-1}
    The function $\rho\colon(0,\infty)\to\real$ is convex if $H>1/2$ and concave if $H< 1/2$.
\item\label{l:cm-2}
    If $H>1/2$, then the function $\rho$ is completely monotone \textup{(}CM\textup{)} on $(1,\infty)$, that is, $\rho\in C^\infty(1,\infty)$ and
    \begin{gather}\label{eq:cm-def}
        (-1)^n \rho^{(n)}(x) \ge 0
        \quad\text{for all\ } n\in\nat\cup\{0\}
        \text{\ and\ } x>1.
    \end{gather}
\item\label{l:cm-3}
    If $H<1/2$, then the function $-\rho$ is completely monotone on $(1,\infty)$.
\end{enumerate}
\end{lemma}

\begin{proof}
\ref{l:cm-1} Using the elementary relation $\frac d{dx}|x|^{2H} = 2H \sgn(x) |x|^{2H-1}$ it is not hard to see that
\begin{align}
    \rho(x)
    &=\notag \frac12\left( |x+1|^{2H} - |x|^{2H}\right) - \frac12\left( |x|^{2H} - |x-1|^{2H} \right)\\
    &=\label{eq:rho}H(2H-1)\int_0^1 \int_{-t}^t |x+s|^{2H-2}\,ds\,dt.
\end{align}
Since $x\mapsto |x+s|^{2H-2}$ is convex, and since convex functions are a convex cone which is closed under pointwise convergence, the double integral appearing in the representation of $\rho(x)$ is again convex. Thus $\rho(x)$ is convex or concave according to $2H-1>0$ or $2H-1<0$, respectively.

\medskip\noindent
\ref{l:cm-2} Let $H>\frac12$ and $x\ge 1$. Then the formula \eqref{eq:rho} remains valid if we replace $|x+s|$ with $(x+s)$. But $(x+s)^{2H-2}$ is CM and so $\rho(x)$ is an integral mixture of CM-functions. Since CM is a convex cone which is closed under pointwise convergence, cf.\ \cite[Cor.~1.6]{ssv2012}, we see that $\rho$ is CM on $(1,\infty)$.

\medskip\noindent
\ref{l:cm-3} The above arguments holds true in the case $H<\frac12$; the only difference is that in this case the factor $(2H-1)$ is negative.
\end{proof}

\begin{remark}
1.
Since $x\mapsto \rho(x+1)$ is a CM-function on $(0.\infty)$, it admits the representation $\rho(x+1) = a+ \int_0^\infty e^{-xt} \,\mu(dt)$, for some positive measure $\mu$ on $[0,\infty)$ and $a\ge0 $, see e.g.\ \cite[Th.~1.4]{ssv2012}. Taking into account that $\rho(+\infty)=0$, it is not hard to see that $a=0$, i.e.
\begin{gather}\label{eq:bernstein-repr}
    \rho(x+1) =  \int_0^\infty e^{-xt} \mu(dt).
\end{gather}

\medskip\noindent
2.
The function $\rho$ can be represented in the form $\rho(x+1) =  \Delta_1^2 f_H(x)$, where $ \Delta_1 f(x) := f(x+1) - f(x)$ is the step-1 difference operator, and $f_H(x) := \frac12 x^{2H}$. Then the second statement of Lemma~\ref{l:cm} follows from the more general result: If $f$ is CM on $(0,\infty)$, then $\Delta_2^1 f$ is CM. Indeed, since CM is a closed convex cone, it is enough to verify the
claim for the \enquote{basic} CM-function $f(x) = e^{-tx}$ where $t\geq 0$ is a parameter. Now we have
\begin{gather*}
    \Delta_1^2 f(x)
    = e^{-(x+2)t} - 2e^{-(x+1)t} + e^{-xt}
    = e^{-xt} (e^{-t}-1)^2,
\end{gather*}
and this is clearly a completely monotone function.

\medskip\noindent
3.
The argument which we used in the proof of Lemma~\ref{l:cm}.\ref{l:cm-2} proves a bit more: The function $x\mapsto\rho(H,x)$ is for $x\geq 1$ and $H>1/2$ even a Stieltjes function, i.e.\ a double Laplace transform. To see this, we note that the kernel $x\mapsto (x+s)^{2H-2}$ is a Stieltjes function. Further details on Stieltjes functions can be found in \cite{ssv2012}.
\end{remark}

As for the following properties, fractional Brownian with Hurst index $H=1$ is degenerate, i.e.\ $B^1_t= t\xi$, where $\xi\sim\Normal(0,1)$; consequently  all $\rho_k=1$ and, the next set of inequalities are equalities. Therefore we consider only $1/2<H<1$.

\begin{corollary}\label{c:rho-properties}
    Let $H\in(\frac12,1)$. The sequence $\{\rho_k,k\ge0\}$ has the following properties
    \begin{enumerate}
    \item\label{c:rho-properties-1}
    Monotonicity and positivity: for any $k\in\nat$
    \begin{gather}\label{eq:rho-monot}
        \rho_{k-1} >\rho_{k} >0.
    \end{gather}

    \item\label{c:rho-properties-2}
    Convexity: for any $k\in\nat$
    \begin{gather}\label{eq:rho-conv}
        \rho_{k-1} - \rho_k > \rho_k - \rho_{k+1}.
    \end{gather}

    \item\label{c:rho-properties-3}
    Log-Convexity: for any $k\in\nat$
    \begin{gather}\label{eq:rho-logconv}
        \rho_{k-1} \rho_{k+1} > \rho_k^2.
    \end{gather}
    \end{enumerate}
\end{corollary}

\begin{proof}
By Lemma \ref{l:cm}, the function $\rho(x)$ is convex on $(0,\infty)$ and completely monotone on $(1,\infty)$; by continunity we can include the endpoints of each interval.

We begin with the observation that a completely monotone function is automatically log-convex. We show this for $\rho$ using the representation \eqref{eq:bernstein-repr}: For any $x\geq 0$
\begin{gather*}
    \rho(x+1) =  \int_0^\infty e^{-xt} \mu(dt),\quad
    \rho'(x+1) = - \int_0^\infty e^{-xt} t \mu(dt),\\
    \rho''(x+1) = \int_0^\infty e^{-xt} t^2 \mu(dt).
\end{gather*}
Thus, the Cauchy--Schwarz inequality yields
\begin{gather}
    \left (\rho'(x)\right )^2 \leq \rho(x) \cdot \rho''(x)
\end{gather}
which guarantees that $x\mapsto \log\rho(x+1)$ is convex.

Therefore all properties claimed in the statement hold for $k\ge2$, convexity even for $k\geq 1$, and we only have to deal with the case $k=1$.

\medskip\noindent\emph{Monotonicity for $k=1$}:
We have to show $\rho_0>\rho_1$. This follows by direct verification, since by \eqref{eq:rho_k},
\begin{gather}\label{eq:g0g1}
    \rho_0=1 \quad\text{and}\quad\rho_1=2^{2H-1}-1
\end{gather}
(recall that $1/2<H<1$).

\medskip\noindent\emph{Log-convexity for $k=1$}:
In this case, the inequality \eqref{eq:rho-logconv} has the form $\rho_2>\rho_1^2$. It immediately follows from the representation \eqref{eq:rho2viarho3} combined with the monotonicity property~\eqref{eq:rho-monot}.
\end{proof}

The previous lemma implies that $\rho_1^2<\rho_2$. The following result gives a sharper bound.
\begin{lemma}\label{l:r12r3}
    If $H\in(\frac12,1)$, then
    \begin{gather}\label{eq:r12lessr3}
        \rho_1^2<\rho_3.
    \end{gather}
\end{lemma}

\begin{proof}
Applying \eqref{eq:rho2viarho3}, we may write
\begin{gather*}
    \rho_3 - \rho_1^2
    = \rho_3 - \rho_2 + \frac12(\rho_1-\rho_3)
    = \frac12(\rho_1 - 2\rho_2 + \rho_3)>0,
\end{gather*}
because of Statement~2 of Corollary~\ref{c:rho-properties}.
\end{proof}

\subsection{Particular cases}\label{ssec:partic}

We will now consider in detail two particular cases: $n=3$ and $n=4$. In these cases we prove the positivity of \emph{all} coefficients $\Gamma_n^k$, establish some relations between them, and study the asymptotic behaviour as $H\uparrow1$. In the case $n=3$ everything is established analytically while in the case $n=4$ the sign of the second coefficient $\Gamma_n^3$ and the relation between the second and the third coefficients, $\Gamma_n^3$ and $\Gamma_n^4$, are verified numerically.

\subsubsection{Case $n=3$}
In the case $n=3$, the system \eqref{eq:system} becomes
\begin{gather}\label{eq:syst-n3}
\left\{\begin{aligned}
\Gamma_3^2 + \Gamma_3^3 \rho_{1} &= \rho_{1} ,\\
\Gamma_3^2 \rho_{1} + \Gamma_3^3 &= \rho_{2} ,
\end{aligned}\right.
\end{gather}
whence
\begin{gather*}
\Gamma_3^2 = \frac{\rho_{1}(1-\rho_{2})}{1-\rho_{1}^2},
\quad
\Gamma_3^3 = \frac{\rho_{2}-\rho_{1}^2}{1-\rho_{1}^2}.
\end{gather*}

\begin{proposition} \label{prop:n3-monot}
For any $H\in(\frac12,1)$,
\begin{gather*}
    \Gamma_3^2> \Gamma_3^3>0.
\end{gather*}
\end{proposition}

\begin{proof}
Recall that, by Corollary~\ref{c:rho-properties} (Statement~1), $1>\rho_{1} > \rho_{2} > \dots$
Hence, the first inequality $\Gamma_3^2>\Gamma_3^3$ is equivalent to
\begin{gather*}
    \rho_{1}(1-\rho_{2})>\rho_{2}-\rho_{1}^2
\quad\text{or}\quad
    (1+\rho_1)(\rho_1-\rho_2)>0,
\end{gather*}
which is true due to Corollary~\ref{c:rho-properties}.

To prove the second inequality $\Gamma_3^3>0$, we need to show that $\rho_2>\rho_1^2$, which was established in Corollary~\ref{c:rho-properties}.
\end{proof}

\begin{remark}
It is worth pointing out that the positivity (and positive definiteness) of the coefficient matrix together with the positivity of the right-hand side of the system does not imply the positivity of solution. Indeed, consider the following system with the same coefficients as in  \eqref{eq:syst-n3}, but another positive right-hand side, say $(b_1,b_2)$:
\begin{gather*}
\begin{pmatrix}
1 & \rho_1 \\
\rho_1 & 1
\end{pmatrix}
\begin{pmatrix}
x_1 \\
x_2
\end{pmatrix}
=
\begin{pmatrix}
b_1 \\
b_2
\end{pmatrix}.
\end{gather*}
The solution has the form:
\begin{gather*}
    x_1 = \frac{b_1-b_2\rho_1}{1-\rho_1^2}, \quad
    x_2 = \frac{b_2-b_1\rho_1}{1-\rho_1^2}.
\end{gather*}
If, for example, $b_2<b_1\rho_1$, then $x_1>0$ and $x_2<0$.
For the system \eqref{eq:syst-n3}, this condition is written as $\rho_2<\rho_1^2$, contradicting Corollary~\ref{c:rho-properties}.
\end{remark}

\begin{proposition}\label{prop:n3-lim}
\begin{align*}
    \lim_{H\uparrow1} \Gamma_3^2
    &= \frac{9 \log 9-8 \log 4}{8 \log 4}
    \approx 0.783083,
\\
    \lim_{H\uparrow1} \Gamma_3^3
    &= \frac{8 \log 16 -9 \log 9}{8 \log 4}
    \approx 0.216917.
\end{align*}
\end{proposition}

\begin{proof}
If we take the limit $H\uparrow 1$ in the relations
\begin{gather}\label{eq:rho12}
\rho_1= 2^{2H-1} - 1,
\quad
\rho_2= \frac12\left(3^{2H} - 2^{2H+1} + 1\right),
\end{gather}
we get
\begin{gather*}
\rho_1 \to 1,
\quad
\rho_2 \to 1,
\quad\text{as } H\uparrow1.
\end{gather*}
Therefore,
\begin{align*}
    \lim_{H\uparrow1}\Gamma_3^2
    &= \lim_{H\uparrow1}\frac{\rho_{1}(1-\rho_{2})}{(1-\rho_{1})(1+\rho_1)}
    =\lim_{H\uparrow1}\frac{1-\rho_{2}}{2(1-\rho_{1})}\\
    &=\lim_{H\uparrow1}\frac{1- \frac12\left(3^{2H} - 2^{2H+1} + 1\right)}{2(1-2^{2H-1} + 1)}\\
    &=\lim_{H\uparrow1}\frac{1- 3^{2H} + 2^{2H+1}}{4(2-2^{2H-1})}
    =\lim_{H\uparrow1}\frac{1- 9^{H} + 2\cdot 4^H}{8-2\cdot 4^H}.
\end{align*}
By l'H\^opital's rule,
\begin{gather*}
    \lim_{H\uparrow1}\Gamma_3^2
    =\lim_{H\uparrow1}\frac{-9^{H}\log 9 + 2\cdot 4^H\log 4}{-2\cdot 4^H\log4}
    =\frac{-9\log 9 + 8\log 4}{-8\log4}.
\end{gather*}
Similarly,
\begin{align*}
    \lim_{H\uparrow1}\Gamma_3^3
    &=\lim_{H\uparrow1} \frac{\rho_{2}-\rho_{1}^2}{1-\rho_{1}^2}
    =\lim_{H\uparrow1} \frac{\rho_{2}-\rho_{1}^2}{2(1-\rho_{1})}\\
    &=\lim_{H\uparrow1} \frac{\frac12\left(3^{2H} - 2^{2H+1} + 1\right)-(2^{2H-1} - 1)^2}{2(2-2^{2H-1})}\\
    &=\lim_{H\uparrow1} \frac{3^{2H} - 2^{2H+1} + 1-2(2^{4H-2} -2^{2H}+ 1)}{4(2-2^{2H-1})}\\
    &=\lim_{H\uparrow1} \frac{9^{H} - 1-\frac12 16^{H}}{8-2\cdot4^{H}}
    =\lim_{H\uparrow1} \frac{9^{H}\log9 -\frac12 16^{H}\log16}{-2\cdot4^{H}\log4}\\
    &=\frac{9\log9 - 8\log16}{-8\log4}.
\qedhere
\end{align*}
\end{proof}

Figure \ref{fig:n=3} shows the dependence of the coefficients $\Gamma_3^2$ and $\Gamma_3^3$ on $H$. It illustrates the theoretical results stated in Propositions~\ref{prop:n3-monot} and~\ref{prop:n3-lim}, in particular, positivity and monotonicity of the coefficients, and convergence to theoretical limit values as $H\uparrow1$.
\begin{figure}
    \centering
    \begin{minipage}{0.47\linewidth}
        \centering
        \includegraphics[width=\linewidth]{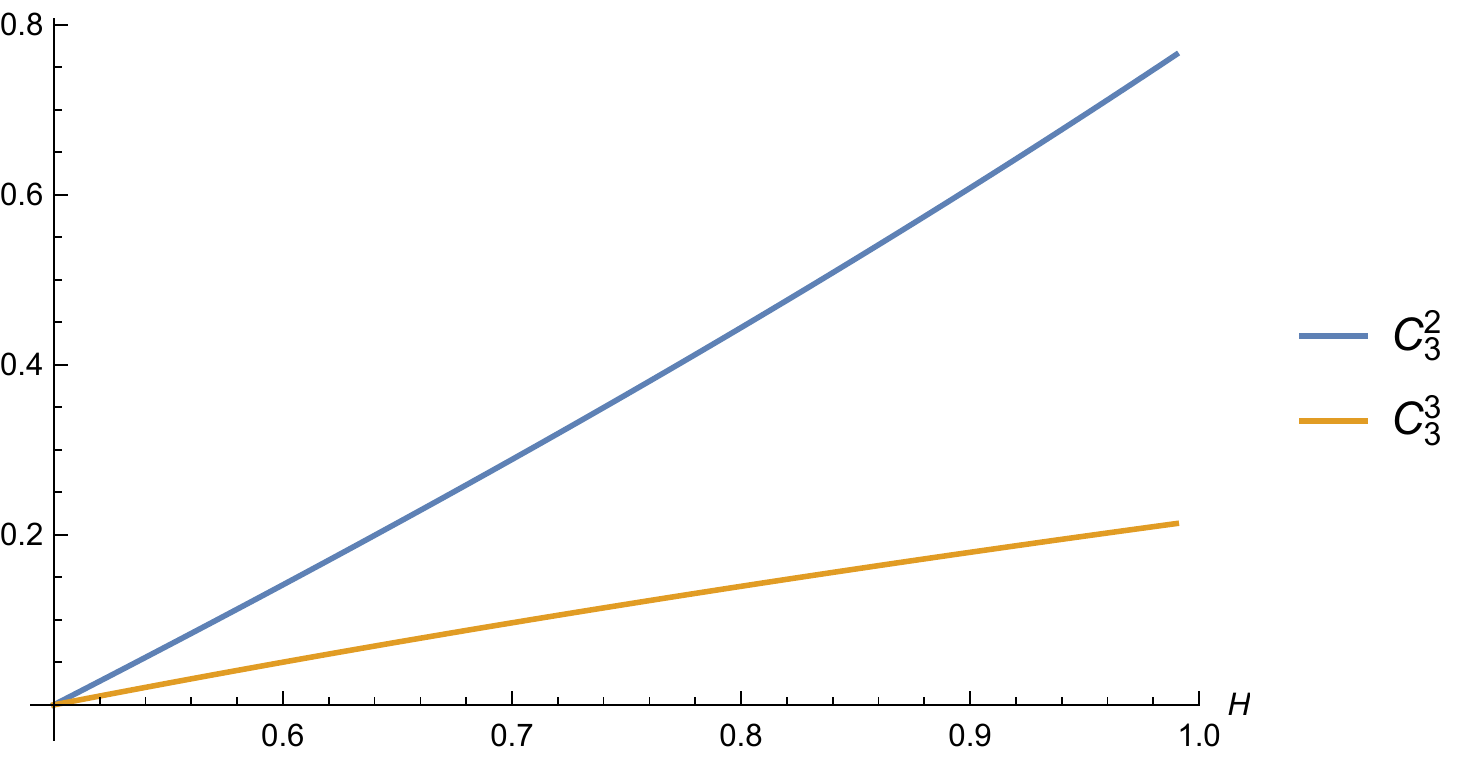} 
        \caption{Case $n=3$: $\Gamma_3^2$ and $\Gamma_3^3$ as the functions of $H$.}\label{fig:n=3}
    \end{minipage}\hfill
    \begin{minipage}{0.47\linewidth}
        \centering
        \includegraphics[width=\linewidth]{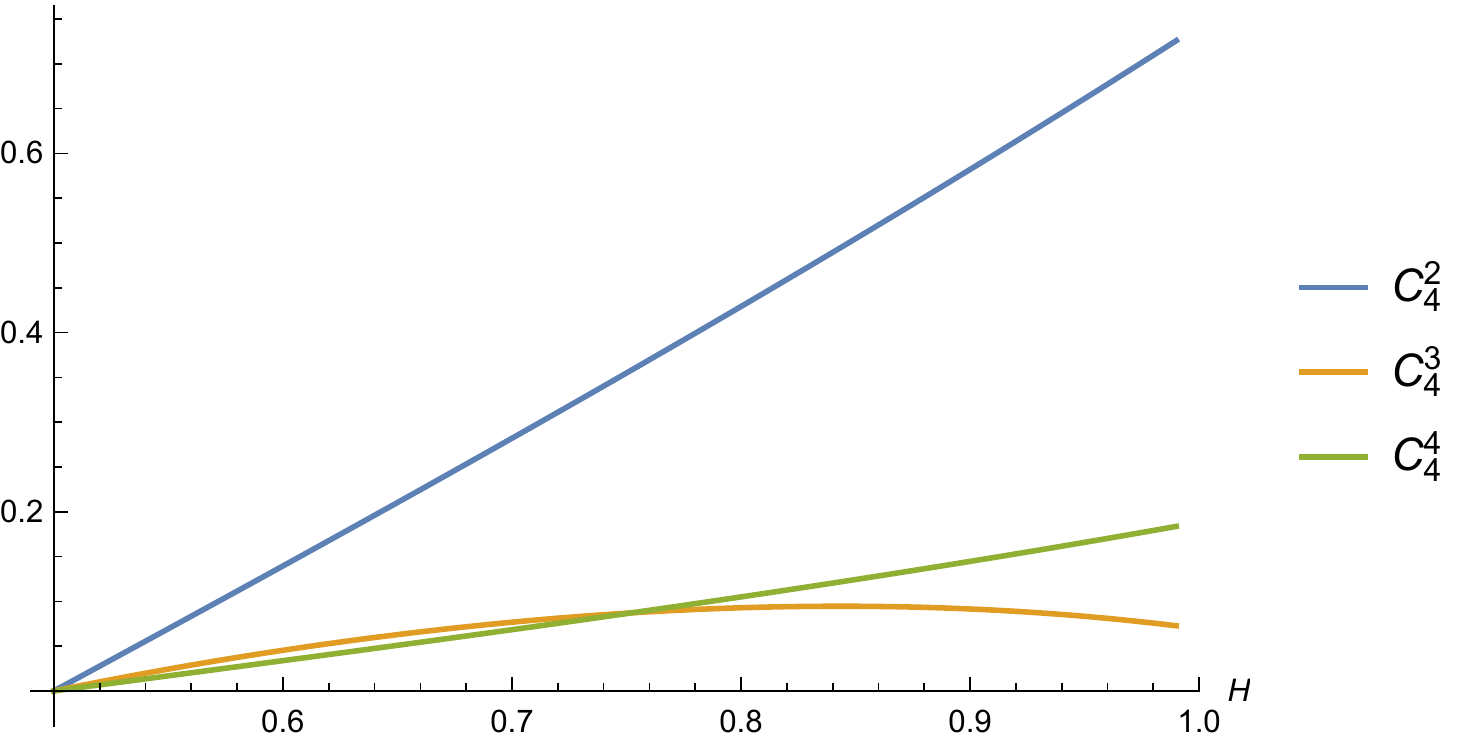} 
        \caption{Case $n=4$: $\Gamma_4^2$, $\Gamma_4^3$, and $\Gamma_4^4$ as the functions of $H$.}\label{fig:n=4}
    \end{minipage}
    \\[12pt]
    \begin{minipage}{0.47\linewidth}
        \centering
        \includegraphics[width=\linewidth]{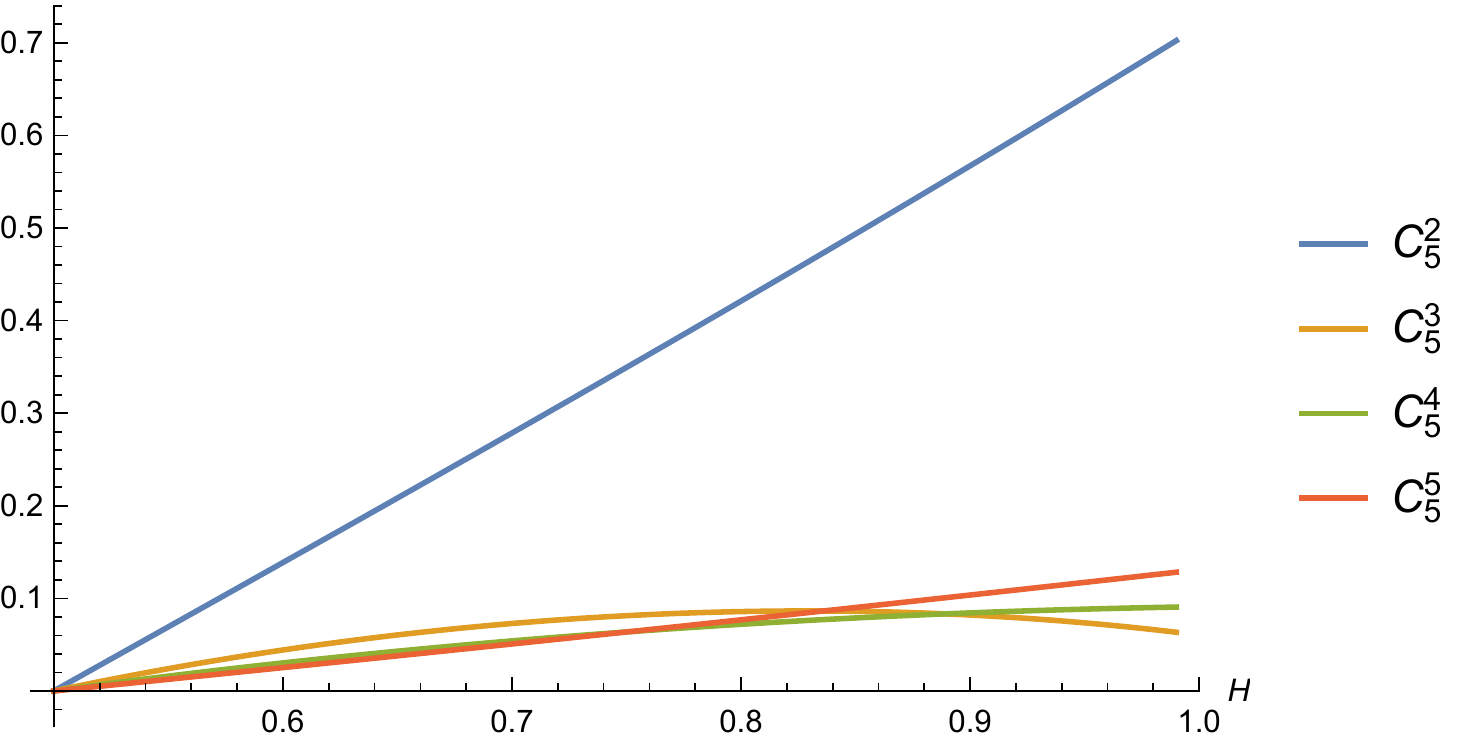} 
        \caption{Case $n=5$: $\Gamma_5^2$, $\Gamma_5^3$, $\Gamma_5^4$, and $\Gamma_5^5$ depending on $H$.}\label{fig:n=5}
    \end{minipage}\hfill
    \begin{minipage}{0.47\linewidth}
        \centering
        \includegraphics[width=\linewidth]{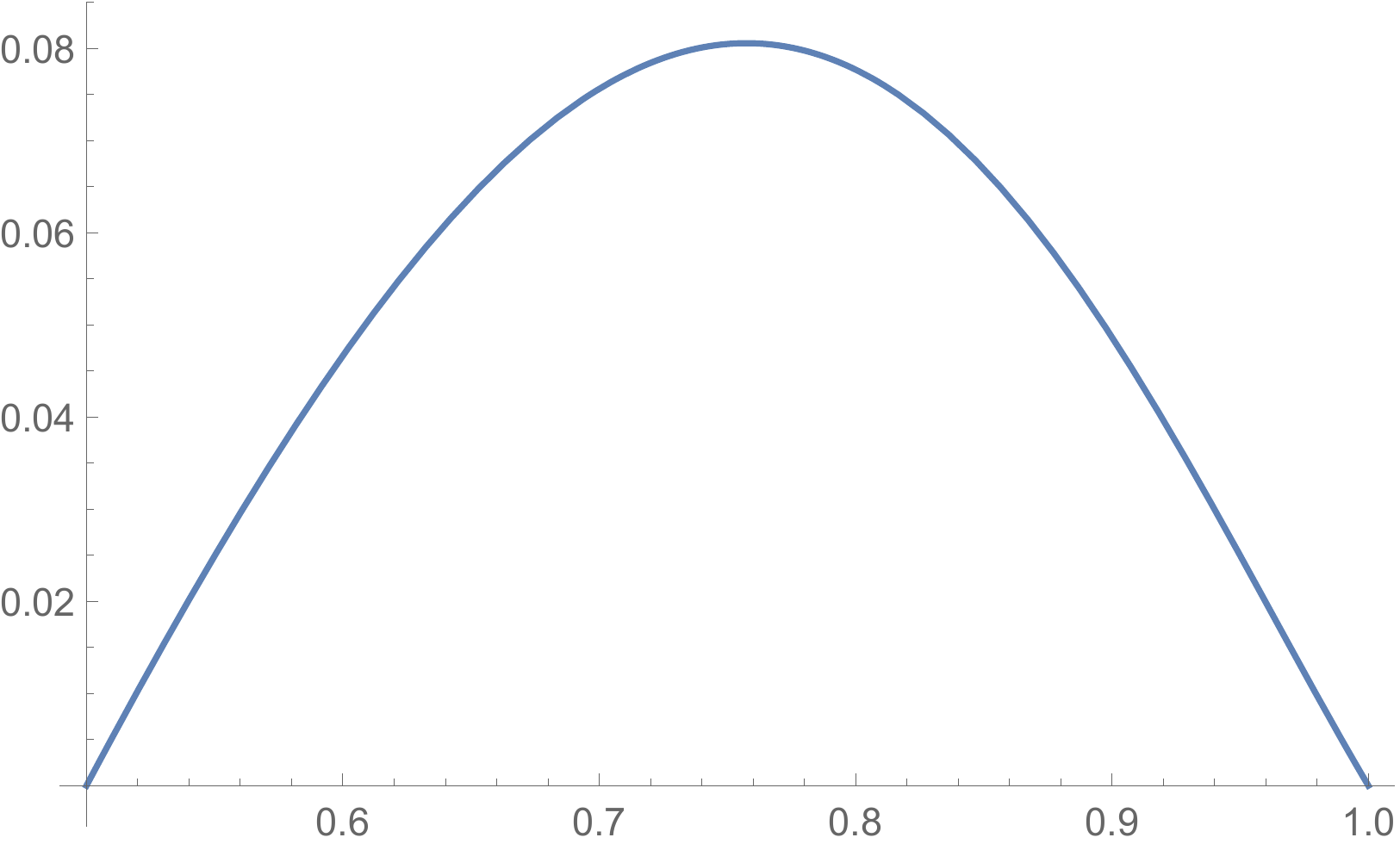} 
        \caption{The left-hand side of \eqref{eq:posit2}.}\label{fig:coef-num}
    \end{minipage}
\end{figure}

\subsubsection{Case $n=4$}
For $n=4$, the system \eqref{eq:system} has the following form
\begin{gather*}
\left\{
\begin{aligned}
\rho_{1} &= \Gamma_4^2 + \Gamma_4^3 \rho_{1} + \Gamma_4^4 \rho_{2},\\
\rho_{2} &= \Gamma_4^2 \rho_{1} + \Gamma_4^3 + \Gamma_4^4 \rho_{1},\\
\rho_{3} &= \Gamma_4^2 \rho_{2} + \Gamma_4^3 \rho_{1}  + \Gamma_4^4.
\end{aligned}
\right.
\end{gather*}
Therefore,
\begin{align}
\label{eq:C42}
    \Gamma_4^2
    &= \frac{\rho_1 + \rho_1^2 \rho_3 + \rho_1\rho_2^2 - \rho_2 \rho_3 -\rho_1^3 - \rho_1 \rho_2}{1  + 2 \rho_1^2\rho_2 - \rho_2^2 - 2 \rho_1^2 },\\
\label{eq:C43}
    \Gamma_4^3
    &= \frac{\rho_1^2 \rho_2 - \rho_2^3 + \rho_1  \rho_2 \rho_3  - \rho_1^2 + \rho_2 - \rho_1 \rho_3}{1  + 2 \rho_1^2\rho_2 - \rho_2^2 - 2 \rho_1^2},\\
\label{eq:C44}
    \Gamma_4^4
    &= \frac{\rho_1^3 + \rho_1 \rho_2^2 - 2\rho_1 \rho_2 + \rho_3 - \rho_1^2 \rho_3}{1  + 2 \rho_1^2\rho_2 - \rho_2^2 - 2 \rho_1^2}.
\end{align}

\begin{proposition} \label{prop:n4-monot}
For any $H\in(\frac12,1)$,
\begin{gather*}
\Gamma_4^2> \Gamma_4^3
\quad\text{and}\quad
\Gamma_4^4>0.
\end{gather*}
\end{proposition}

\begin{proof}
The positivity of the denominator follows from the representation
\begin{gather}\label{eq:n4denom}
    1  + 2 \rho_1^2\rho_2 - \rho_2^2 - 2 \rho_1^2
    = (1-\rho_2) (1 -  \rho_1^2) + (1-\rho_2)(\rho_2 - \rho_1^2)
\end{gather}
and with Corollary~\ref{c:rho-properties}. Therefore, it suffices to prove the claimed relations for the numerators of $\Gamma_4^2$, $\Gamma_4^3$, and
$\Gamma_4^4$.

\medskip\noindent
1. Let us prove that $\Gamma_4^2 > \Gamma_4^3$. The difference between the numerators of $\Gamma_4^2$ and $\Gamma_4^3$ is equal to
\begin{gather*}
    (\rho_1 + \rho_1^2 \rho_3 + \rho_1\rho_2^2 - \rho_2 \rho_3 -\rho_1^3 - \rho_1 \rho_2)
    - (\rho_1^2 \rho_2 - \rho_2^3 + \rho_1  \rho_2 \rho_3  - \rho_1^2 + \rho_2 - \rho_1 \rho_3)
    \\
    =(\rho_1-\rho_2)(1+\rho_3)(1-\rho_2) + (\rho_1^2-\rho_2^2)(1-\rho_1-\rho_2+\rho_3)>0,
\end{gather*}
since $\rho_1>\rho_2$ and $1-\rho_1\ge\rho_2-\rho_3$ by Statements~1 and~2 of Corollary~\ref{c:rho-properties}.

\medskip\noindent
2. Finally, the positivity of $\Gamma_4^4$ follows from the following representation of its numerator:
\begin{gather*}
    \rho_1^3 + \rho_1 \rho_2^2 - 2\rho_1 \rho_2 + \rho_3 - \rho_1^2 \rho_3
     = (\rho_1-\rho_2)^2 + (1-\rho_1)(\rho_3-\rho_1^2+\rho_1\rho_3-\rho_2^2),
\end{gather*}
because $\rho_3>\rho_1^2$ and $\rho_1\rho_3>\rho_2^2$ by \eqref{eq:r12lessr3}, and \eqref{eq:rho-logconv}, respectively.
\end{proof}

Figure \ref{fig:n=4} confirms the above proposition. We see that $\Gamma_4^2$ is the largest coefficient. However, $\Gamma_4^3>\Gamma_4^4$ only for $H < 0.752281$, for larger $H$ the order changes.

\begin{remark}
Consider numerically the relation between $\Gamma_4^3$  and $\Gamma_4^4$ and the sign of $\Gamma_4^3$. One may represent the numerator of $\Gamma_4^3$ as follows:
\begin{gather}\label{eq:C43-numer}
    \rho_1^2 \rho_2 - \rho_2^3 + \rho_1  \rho_2 \rho_3  - \rho_1^2 + \rho_2 - \rho_1 \rho_3
    =(1-\rho_2)(\rho_2+\rho_2^2-\rho_1^2-\rho_1\rho_3).
\end{gather}
Thus we need to establish that
\begin{gather}\label{eq:posit2}
    \rho_2+\rho_2^2-\rho_1^2-\rho_1\rho_3 > 0.
\end{gather}
We established this fact numerically, since we could not come up with an analytical proof. Figure~\ref{fig:coef-num} shows the plot of the left-hand side of \eqref{eq:posit2} that confirms the positivity of $\Gamma_4^3$.

However, we can look at  \eqref{eq:posit2} from another point of view. Rewrite \eqref{eq:posit2} in the following form:
\begin{gather*}
    \frac{1+\rho_2}{\rho_1}
    > \frac{\rho_1+\rho_3}{\rho_2}.
\end{gather*}
The left- and the right-hand sides of this inequality are the values at the points $x=0$ and $x=1$, respectively, of the following function:
\begin{align*}
    \psi(H, x)&\coloneqq\frac{\rho(H,x) + \rho(H,x+2)}{\rho(H,x+1)}\\
    &=\frac{(x + 3)^{2 H} - 2 (x + 2)^{2 H} + 2 (x + 1)^{2 H} -2 x^{2 H} + (1 - x)^{2 H}}
           {(x + 2)^{2 H} - 2 (x + 1)^{2 H} + x^{2 H}},
    \quad  x\in[0,1].
\end{align*}
The graph of the surface $\set{\psi(H, x),\, x\in[0,1],\, H\in(1/2,1)}$ is shown in Figure~\ref{fig:surf}. It was natural to \emph{assume} that the function $\psi(H, x)$ decreases in $x$ for any $H$, being at $x=0$ bigger than at $x=1$. However, the function is not monotone for all $H$. Figure~\ref{fig:psi} contains two-dimensional plots of $\set{\psi(H, x),\, x\in[0,1]}$ for four different values of $H$: $0.6$, $0.7$, $0.8$ and $0.9$. We observe that $\psi(H, 0)>\psi(H, 1)$ for each value of $H$ however, the function $\psi(H, x)$ changes its behaviour from increasing to decreasing.
\end{remark}

\begin{remark}
The unexpected behaviour of ${\psi(H, x), x\in[0,1]}$ (first increasing, then decreasing) is a consequence of the non-standard term $(1 - x)^{2 H}$. For $x\ge 1$ this function decreases in $x$ for any $H>1/2$. Indeed, for $x\ge 1$ it has a form
\begin{align*}
    \psi(H, x)
    &= \frac{(x+3)^{2H} - 2(x+2)^{2H} + 2(x+1)^{2H} - 2x^{2H} + (x-1)^{2H}}
            {(x+2)^{2H} - 2(x+1)^{2H}+x^{2H}}\\
    &= -2 + \frac{(x+3)^{2H} + (x-1)^{2H} - 2(x+1)^{2H}} {(x+2)^{2H} - 2(x+1)^{2H}+x^{2H}}\\
    &= -2 + \frac{(1+\frac{2}{x+1})^{2H} + (1-\frac{2}{x+1})^{2H} - 2} {(1+\frac{1}{x+1})^{2H} + (1-\frac{1}{x+1})^{2H} -2}.
\end{align*}
Write $y=\frac{1}{x+1}\in(0,\frac12]$. It is sufficient to prove that the function
\begin{gather*}
    \eta(H,y) = \frac{(1+2y)^{2H} + (1-2y)^{2H} - 2} {(1+y)^{2H} + (1-y)^{2H} -2},
    \quad y\in(0,\tfrac12],
\end{gather*}
increases in $y$ for any $H\in(\frac12,1)$. However, for $y<\frac12$,
\begin{gather*}
    (1+y)^{2H}+(1-y)^{2H}-2 = \sum_{k=0}^\infty c_{k} y^{2k+2}
\intertext{and}
(1+2y)^{2H}+(1-2y)^{2H}-2 = \sum_{k=0}^\infty c_{k}\, (2y)^{2k+2},
\intertext{where}
    c_{k} = \frac{4H(2H-1)(2H-2)\dots(2H-2k-1)}{(2k+2)!}
    =\frac{2(2H)_{2k+2}}{(2k+2)!}, \quad k=0,1,2,\dots
\end{gather*}
(here $(x)_n = x (x-1)\dots (x-n+1)$ is the Pochhammer symbol).
The monotonicity of $\eta(H,y)$ for $y\in(0,\frac12]$ can be proved by differentiation. Then
\begin{gather}\label{eq:eta-series}
    \eta(H,y) = \frac{\sum_{k=0}^\infty c_{k}\, 2^{2k+2}y^{2k}}{\sum_{k=0}^\infty c_{k} y^{2k}},
\end{gather}
hence, the partial derivative equals
\begin{align*}
    &\frac{\partial}{\partial y}\eta(H,y)\\
    &= \left(\sum_{k=0}^\infty c_{k} y^{2k}\right)^{-2}
    \left(\sum_{k=1}^\infty c_{k}\, 2^{2k+2}2ky^{2k-1}\sum_{l=0}^\infty c_{l} y^{2l}-\sum_{k=1}^\infty c_{k} 2k y^{2k-1}\sum_{l=0}^\infty c_{l}\, 2^{2l+2}y^{2l}\right)\\
    &= \left(\sum_{k=0}^\infty c_{k} y^{2k}\right)^{-2} \sum_{k=1}^\infty \sum_{l=0}^\infty c_{k}c_{l} \cdot2k\left(2^{2k+2} - 2^{2l+2}\right)y^{2k+2l-1}.
\end{align*}
By rearranging the double sum in the numerator, we get the expression
\begin{gather*}
    \frac{\partial}{\partial y}\eta(H,y)
    = \left(\sum_{k=0}^\infty c_{k} y^{2k}\right)^{-2} \sum_{k=1}^\infty \sum_{l=0}^k c_{k}c_{l}(2k-2l)\left(2^{2k+2} - 2^{2l+2}\right)y^{2k+2l-1},
\end{gather*}
which is clearly positive. Thus for any $H\in(\frac12,1)$, $\eta(H,y)$ is increasing as a function of $y\in(0,\frac12)$.

Let us try to establish a bit more. We can represent $\eta(H,y)$ in the following form
\begin{gather*}
\eta(H,y) = \frac{\sum_{k=0}^\infty c_{k}\, 2^{2k+2}y^{2k}}{\sum_{k=0}^\infty c_{k} y^{2k}} = \sum_{k=0}^\infty b_{k} y^{2k},
\end{gather*}
where the coefficients $b_k$ can be found successively from the following equations:
\begin{align*}
    2^2 c_0 &= c_0 b_0,\\
    2^4 c_1 &= c_0 b_1 + c_1 b_0,\\
    2^6 c_2 &= c_0 b_2 + c_1 b_1 + c_2 b_0,\\
    2^8 c_3 &= c_0 b_3 + c_1 b_2 + c_2 b_1 + c_3 b_0,\\
            &\dots
\end{align*}
Let us find the first few coefficients: $b_0= 2^2=4$,
\begin{align*}
    b_1 &= \frac{2^4 c_1 -  2^2 c_1}{c_0}
         =\frac{(2^4 -  2^2)\,\frac{(2H-1)(2H-2)(2H-3)}{4!}}{\frac{(2H-1)}{2!}}
         = (2H-2)(2H-3),\\
    b_2 &= \frac{(2^6 -  2^2) c_2 - c_1 b_1}{c_0}\\
        &= \frac{60\,\frac{(2H-1)(2H-2)(2H-3)(2H-4)(2H-5)}{6!}-\frac{(2H-1)(2H-2)^2(2H-3)^2}{4!}}{\frac{(2H-1)}{2!}}\\
        &= \frac{2!(2H-2)(2H-3)}{4!}\bigl(2 (2H-4)(2H-5) - (2H-2)(2H-3)\bigr)\\
        &= \frac16 (2H-2)(2H-3) (2H^2-13H+17).
\end{align*}
It is easy to see that $b_0$, $b_1$, and $b_2$ are positive for $H\in(\frac12,1)$.
We believe that $b_k>0$ for all $k$.
However, the proof of this fact remains an open problem.
\end{remark}

\begin{figure}
    \centering
    \begin{minipage}[b]{0.47\textwidth}
        \centering
        \includegraphics[width=\textwidth]{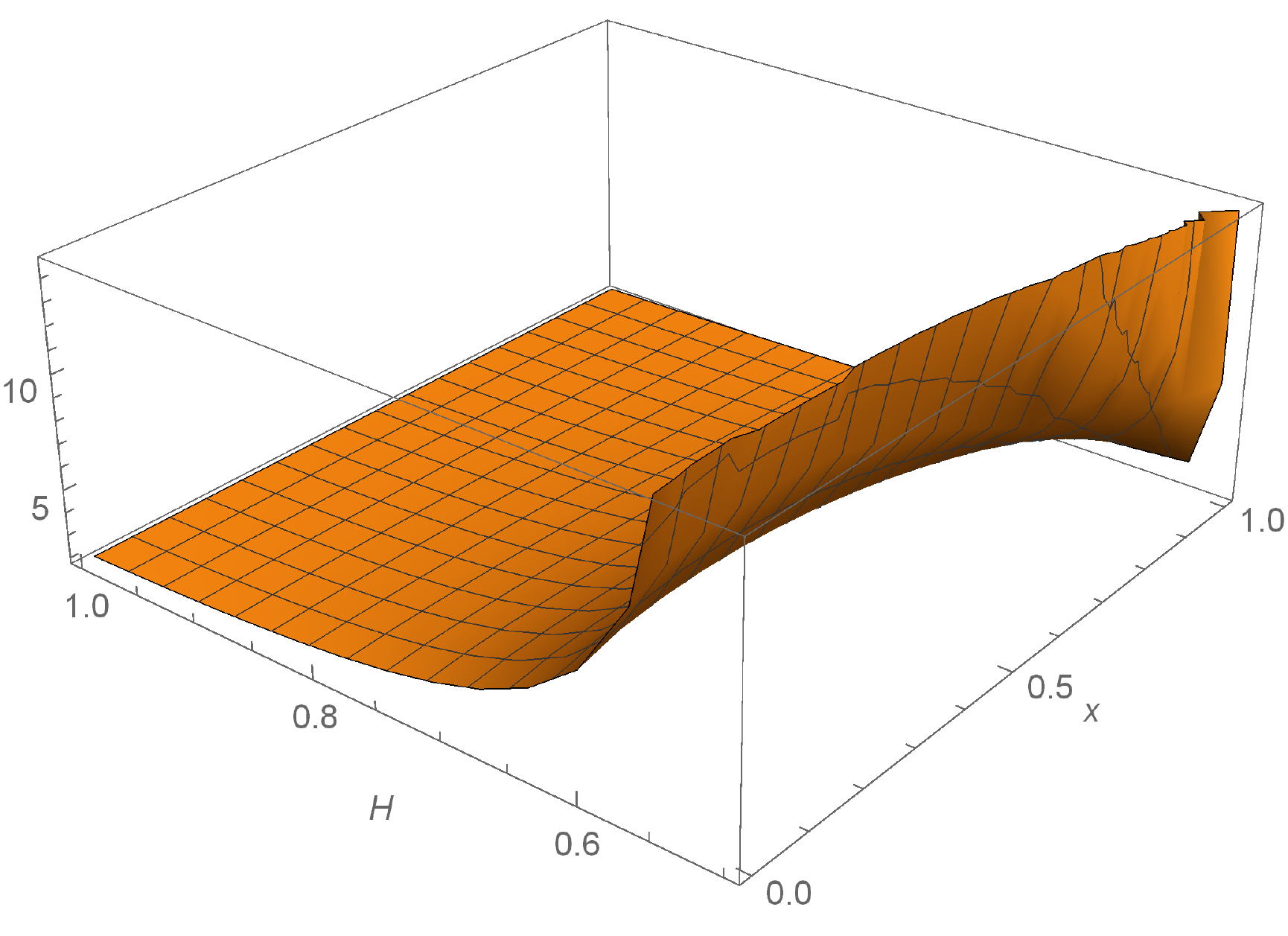} 
        \caption{The function $\psi(H,x)$.}\label{fig:surf}
    \end{minipage}\hfill
    \begin{minipage}[b]{0.47\textwidth}
        \centering
        \includegraphics[width=\textwidth]{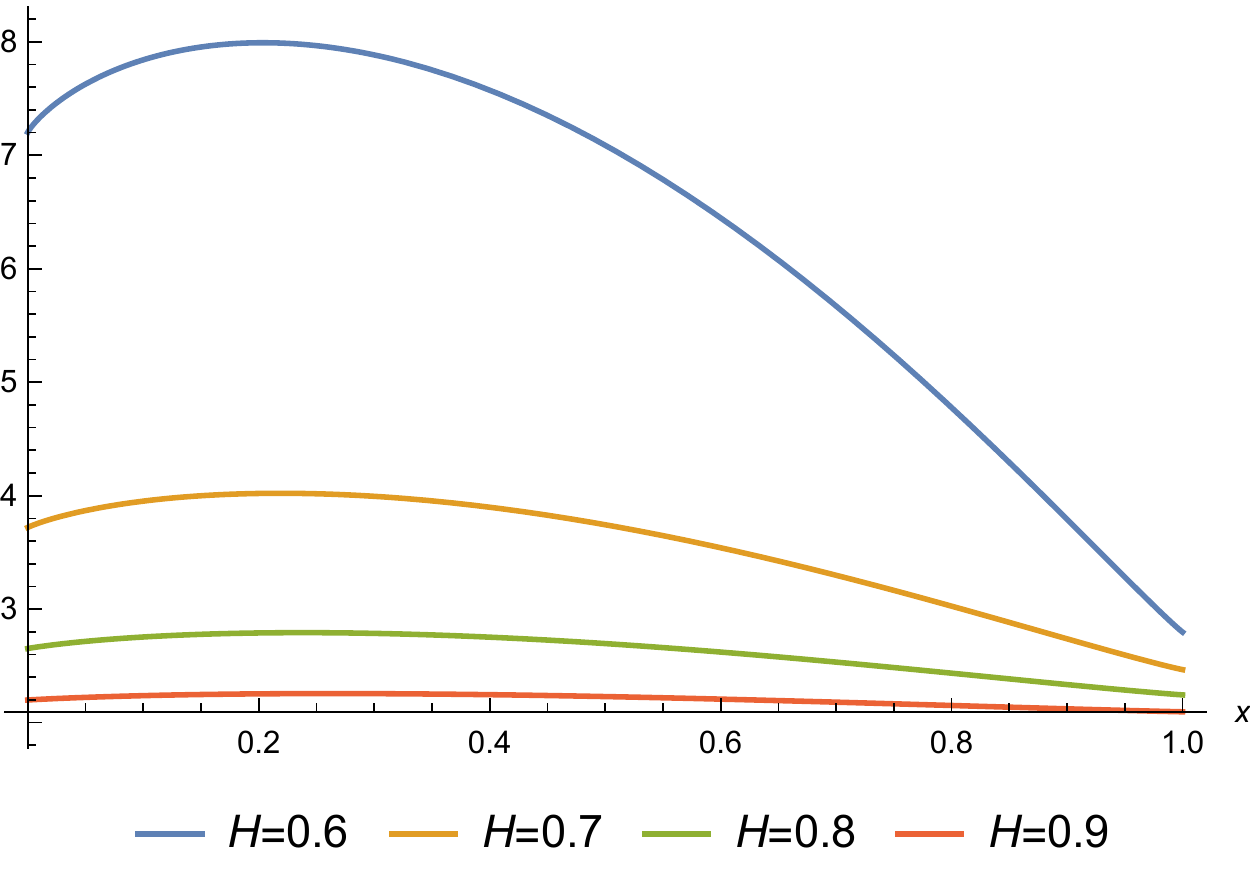} 
        \caption{The function $\psi(H, x)$.}\label{fig:psi}
    \end{minipage}
\end{figure}

\begin{proposition}\label{prop:n4-lim}
\begin{align*}
\lim_{H\uparrow1} \Gamma_4^2 &= \frac{531 \log^2 4 + 72 \log^2 6 + 51 \log^2 9 - 384 \log^2 12 + 108 \log^2 18}{96 \log^2 12 - 640 \log^2 2 - 51\log^2 9}
\approx 0.742250,
\\
\lim_{H\uparrow1} \Gamma_4^3 &= \frac{48 \log 2 -15 \log 9}{16 \log 2 - 3 \log 9}
\approx 0.069508,
\\
\lim_{H\uparrow1} \Gamma_4^4 &= \frac{108 \log^2 18-364 \log^2 2 - 216 \log 2 \log 9 - 81 \log^2 9}{96\log^2 12 - 640 \log^2 2 - 51 \log^2 9}
\approx 0.188242.
\end{align*}
\end{proposition}
\begin{remark}
 Obviously, the sum of the limits of the coefficients is $1$, as expected.
\end{remark}
\begin{proof}[Sketch of proof]
The proof is straightforward. Substituting $\rho_1= 2^{2H-1} - 1$, $\rho_2= \frac12\left(3^{2H} - 2^{2H+1} + 1\right)$, and $\rho_3 = \frac12\left(4^{2H} - 2\cdot 3^{2H} + 2^{2H}\right)$ into \eqref{eq:C42}--\eqref{eq:C44}, and simplifying the resulting expressions, we get
\begin{footnotesize}
\begin{align*}
    \Gamma_4^2
    &= \frac{2-7\cdot 4^H+4^{3 H+1}+6^{2 H+1}-4\cdot 9^H-2^{4 H+3}\cdot 9^H+2\cdot 9^{2H} +4^{4H}+18^{2H}}{2 \left(-2^{6 H+1}+2\cdot 9^H + 3\cdot 2^{4H} - 9^{2H} + 12^{2H}-1\right)},
\\
    \Gamma_4^3
    &= \frac{-2^{2 H+1}+4^{2 H+1}-2^{2 H+1}\cdot 9^H - 64^H + 9^{2H} - 1}{2 \left(2^{2 H+1}+9^H-16^H-1\right)},
\\
    \Gamma_4^4
    &= \frac{-3\cdot 2^{4 H+1} + 4^H + 8^{2 H+1} + 4\cdot 9^H - 2^{2 H+1} \cdot 9^H - 2^{4 H+1} \cdot 9^H - 2\cdot 9^{2H} - 4^{4H} + 18^{2H} - 2}{2 \left(-2^{6 H+1}+2\cdot 9^H + 3\cdot 2^{4H} - 9^{2H} + 12^{2H}-1\right)}
\end{align*}
\end{footnotesize}%
(for $\Gamma_4^3$ we first cancel out the factor $1-\rho_2$, see \eqref{eq:n4denom} and \eqref{eq:C43-numer}). Then applying l'H\^opital's rule (twice) we arrive at the claimed limits by simple algebra.
\end{proof}

\subsubsection{Case $n=5$}
For $n=5$, we present graphical results only, see Figure~\ref{fig:n=5}. The situation here is more complicated compared to the case $n=4$. The first coefficient $\Gamma_5^2$ is still the largest, however the order of three other coefficients changes several times depending on $H$. In particular, for $H$ close to 1/2 these coefficients are decreasing, but for $H$ close to 1 they are increasing.

\subsection{Recurrence relations for the coefficients}\label{ssec:recurrent}
In general, there are several ways how to get \eqref{eq:system}. For example, we can consider the coefficients $\Gamma_n^k$ as a result of minimizing the value of the quadratic form
\begin{gather*}
    \Ee(\Delta_1-\sum_{k=2}^n \alpha_k\Delta_k)^2.
\end{gather*}
Evidently, differentiation leads again to the system \eqref{eq:system}. We can look for the coefficients with the help of the inverse matrix $A^{-1}$, where $A$ is from \eqref{matrices}. However, to calculate the entries of the inverse matrix, is as difficult as to calculate the determinants. But it is possible to  avoid determinants using the properties of fGn. More precisely,  we propose a recurrence method to calculate the coefficients $\Gamma_n^k$ successively, starting with $\Gamma_2^2 = \rho_{1} = 2^{2H-1} -1$.

\begin{proposition}
The following relations hold true:
\begin{align}
    \Gamma_{n+1}^{n+1}
    &= \frac{\rho_{n} - \sum_{k=2}^n \Gamma_n^k \rho_{n+1-k}}{1 - \sum_{k=2}^n \Gamma_n^k \rho_{k-1}},
    \quad n\ge2,
\label{eq:Cn+1n+1}
\\
    \Gamma_{n+1}^k
    &= \Gamma_n^k - \Gamma_{n+1}^{n+1} \Gamma_n^{n-k+2},
    \quad n\ge2,\; 2\le k\le n.
\label{eq:Cn+1k}
\end{align}
\end{proposition}

\begin{proof}
In order to prove \eqref{eq:Cn+1n+1} and \eqref{eq:Cn+1k}, we use the theorem on normal correlation as well as the independence of
$\Delta_{n+1} - \Ee\left(\Delta_{n+1} \mid \Delta_2,\dots,\Delta_n\right)$ and any of $\Delta_k, 2\le k\le n$. We get
\begin{gather}\label{eq:proj2}
    \Ee \left(\Delta_1 \mid \Delta_2,\dots,\Delta_n,\Delta_{n+1}\right)
    = \sum_{k=2}^n \widetilde \Gamma_n^k \Delta_k + \Gamma_{n+1}^{n+1} \left(\Delta_{n+1} - \Ee\left(\Delta_{n+1} \mid \Delta_2,\dots,\Delta_n\right)\right),
\end{gather}
where $\widetilde \Gamma_n^k$, $2\le k\le n$, are some constants. Now we take the conditional expectation $\Ee \left(\cdot\mid \Delta_2,\dots,\Delta_n\right)$ on both sides of \eqref{eq:proj2} to get
\begin{gather*}
    \Ee \left(\Delta_1 \mid \Delta_2,\dots,\Delta_n\right)
    = \sum_{k=2}^n \widetilde \Gamma_{n+1}^k \Delta_k.
\end{gather*}
Comparing this equality with \eqref{eq:proj1}, and taking into account that the increments  $\Delta_k, 2\le k\le n$ are linearly independent,
we conclude that
\begin{gather*}
    \sum_{k=2}^n \widetilde \Gamma_{n+1}^k \Delta_k
    = \sum_{k=2}^n \Gamma_n^k \Delta_k.
\end{gather*}
Now we insert this equality into \eqref{eq:proj2} and see
\begin{gather}\label{eq:eq1}
    \Ee \left(\Delta_1 \mid \Delta_2,\dots,\Delta_n,\Delta_{n+1}\right)
    = \sum_{k=2}^n \Gamma_n^k \Delta_k
        + \Gamma_{n+1}^{n+1} \left(\Delta_{n+1} - \Ee\left(\Delta_{n+1} \mid \Delta_2,\dots,\Delta_n\right)\right).
\end{gather}
After multiplying both sides of the last equality by $\left(\Delta_{n+1} - \Ee\left(\Delta_{n+1} \mid \Delta_2,\dots,\Delta_n\right)\right)$ and taking expectations, we arrive at
\begin{gather*}
    \Ee \left(\Delta_1 \left(\Delta_{n+1} - \Ee\left(\Delta_{n+1} \mid \Delta_2,\dots,\Delta_n\right)\right)\right)
    = \Gamma_{n+1}^{n+1} \Ee\left(\Delta_{n+1} - \Ee\left(\Delta_{n+1} \mid \Delta_2,\dots,\Delta_n\right)\right)^2.
\end{gather*}
It follows from the stationarity of the increments that the indices $n+1$ and $1$ of the last equality play symmetric roles, i.e.\ it is equivalent to
\begin{gather*}
    \Ee \left(\Delta_{n+1} \left(\Delta_{1} - \Ee\left(\Delta_{1} \mid \Delta_2,\dots,\Delta_n\right)\right)\right)
    = \Gamma_{n+1}^{n+1} \Ee\left(\Delta_{1} - \Ee\left(\Delta_{1} \mid \Delta_2,\dots,\Delta_n\right)\right)^2.
\end{gather*}
From this we conclude that
\begin{align*}
    \Gamma_{n+1}^{n+1}
    &= \frac{\Ee \left(\Delta_{n+1} \left(\Delta_{1} - \Ee\left(\Delta_{1} \mid \Delta_2,\dots,\Delta_n\right)\right)\right)}
        {\Ee\left(\Delta_{1} - \Ee\left(\Delta_{1} \mid \Delta_2,\dots,\Delta_n\right)\right)^2}
\\
    &= \frac{\rho_{n} - \sum_{k=2}^n \Gamma_n^k \rho_{n+1-k}}{1 - \sum_{k=2}^n \Gamma_n^k \rho_{k-1}}.
\end{align*}
Thus the relation \eqref{eq:Cn+1n+1} is proved.

Using again the symmetry of the stationary increments, it is not hard to see that
\begin{gather*}
    \Ee \left(\Delta_{n+1} \mid \Delta_2,\dots,\Delta_n\right)
    = \sum_{k=2}^n \Gamma_n^{n-k+2} \Delta_k.
\end{gather*}
Therefore, we get from \eqref{eq:eq1} that
\begin{align*}
    \Ee \left(\Delta_1 \mid \Delta_2,\dots,\Delta_n,\Delta_{n+1}\right)
    &= \sum_{k=2}^n \Gamma_n^k \Delta_k + \Gamma_{n+1}^{n+1} \Delta_{n+1} - \Gamma_{n+1}^{n+1} \sum_{k=2}^n \Gamma_n^{n-k+2} \Delta_k
\\
    &= \sum_{k=2}^n \left(\Gamma_n^k - \Gamma_{n+1}^{n+1}\Gamma_n^{n-k+2}\right)\Delta_k + \Gamma_{n+1}^{n+1} \Delta_{n+1},
\end{align*}
and \eqref{eq:Cn+1k} follows.
\end{proof}

\subsection{Positivity of $\Gamma_n^2$}\label{ssec:first-pos}
We conjecture that all coefficients $\Gamma_n^k$, $2\le k\le n$, are positive.
However, analytically we can prove only the positivity of the leading coefficient, $\Gamma_n^2$.
\begin{proposition}
    For all $n\ge1$, $\Gamma_{n+1}^2>0$.
\end{proposition}
\begin{proof}
From the stationarity of the increments it follows that
\begin{gather*}
    \Ee \left(\Delta_2 \mid \Delta_3,\dots,\Delta_{n+1}\right)
    = \sum_{k=3}^{n+1} \Gamma_n^{k-1} \Delta_k.
\end{gather*}
Similarly to \eqref{eq:eq1},
\begin{gather*}
    \Ee \left(\Delta_1 \mid \Delta_2,\dots,\Delta_n,\Delta_{n+1}\right)
    = \widetilde \Gamma_{n+1}^{2} \left(\Delta_{2} - \Ee\left(\Delta_{2} \mid \Delta_3,\dots,\Delta_{n+1}\right)\right)
    + \sum_{k=3}^{n+1} \Gamma_{n+1}^{k} \Delta_k,
\end{gather*}
and so
\begin{gather*}
    \Gamma_{n+1}^{2}
    = \widetilde \Gamma_{n+1}^{2}
    = \frac{\Ee\left(\left(\Delta_{2} - \Ee\left(\Delta_{2} \mid \Delta_3,\dots,\Delta_{n+1}\right)\right)\Delta_1\right)}
           {\Ee\left(\Delta_{2} - \Ee\left(\Delta_{2} \mid \Delta_3,\dots,\Delta_{n+1}\right)\right)^2}.
\end{gather*}
It remains to prove the positivity of the numerator
\begin{gather*}
    \Ee\left(\left(\Delta_{2} - \Ee\left(\Delta_{2} \mid \Delta_3,\dots,\Delta_{n+1}\right)\right)\Delta_1\right)
    = \rho_1 - \sum_{k=3}^{n+1} \Gamma_n^{k-1} \rho_{k-1}.
\end{gather*}
But we know from \eqref{eq:system} that
\begin{gather*}
    \rho_{1} = \sum_{k=2}^n \Gamma_n^k \rho_{k-2}.
\end{gather*}
Therefore,
\begin{gather*}
    \rho_{1} - \sum_{k=3}^{n+1} \Gamma_n^{k-1} \rho_{k-1}
    =\sum_{k=2}^n \Gamma_n^k \rho_{k-2} - \sum_{k=2}^{n} \Gamma_n^k \rho_{k}
    = \sum_{k=2}^n \Gamma_n^k \left(\rho_{k-2} - \rho_{k}\right)
    >0,
\end{gather*}
since the sequence $\rho_{k}$ is decreasing, see Corollary~\ref{c:rho-properties}.
\end{proof}

\section{Numerical results}
\label{section3}

\subsection{Properties of coefficients. Positivity and (non)monotonicity}

In this section we compute numerically the coefficients $\Gamma_n^k$ for various values of $H$. In Tables \ref{tab:0.51}--\ref{tab:0.99} the results for $H = 0.51$, $0.6$, $0.7$, $0.8$, $0.9$, and $0.99$ are listed for $2\le n\le 10$.

\begin{small}
\begin{table}
\caption{Coefficients $\Gamma_n^k$ for $H=0.51$}\label{tab:0.51}
\centering
\begin{tabular}{c|ccccccccc}
  \hline
$n\backslash k$ & 2 & 3 & 4 & 5 & 6 & 7 & 8 & 9 & 10\\
 \hline
2 & 0.01396 &  &  &  &  &  &  &  &  \\
  3 & 0.01389 & 0.00521 &  &  &  &  &  &  &  \\
  4 & 0.01387 & 0.00516 & 0.00339 &  &  &  &  &  &  \\
  5 & 0.01386 & 0.00515 & 0.00336 & 0.00253 &  &  &  &  &  \\
  6 & 0.01386 & 0.00514 & 0.00335 & 0.00250 & 0.00201 &  &  &  &  \\
  7 & 0.01385 & 0.00514 & 0.00334 & 0.00249 & 0.00199 & 0.00167 &  &  &  \\
  8 & 0.01385 & 0.00514 & 0.00334 & 0.00248 & 0.00198 & 0.00165 & 0.00143 &  &  \\
  9 & 0.01385 & 0.00513 & 0.00334 & 0.00248 & 0.00198 & 0.00165 & 0.00142 & 0.00125 &  \\
  10 & 0.01385 & 0.00513 & 0.00333 & 0.00248 & 0.00198 & 0.00164 & 0.00141 & 0.00124 & 0.00111 \\
   \hline
\end{tabular}
\end{table}

\begin{table}
\caption{Coefficients $\Gamma_n^k$ for $H=0.6$}\label{tab:0.6}
\centering
\begin{tabular}{c|ccccccccc}
  \hline
$n\backslash k$ & 2 & 3 & 4 & 5 & 6 & 7 & 8 & 9 & 10\\
  \hline
2 & 0.14870 &  &  &  &  &  &  &  &  \\
  3 & 0.14123 & 0.05020 &  &  &  &  &  &  &  \\
  4 & 0.13954 & 0.04542 & 0.03383 &  &  &  &  &  &  \\
  5 & 0.13868 & 0.04427 & 0.03031 & 0.02522 &  &  &  &  &  \\
  6 & 0.13817 & 0.04366 & 0.02942 & 0.02243 & 0.02013 &  &  &  &  \\
  7 & 0.13784 & 0.04329 & 0.02893 & 0.02170 & 0.01781 & 0.01675 &  &  &  \\
  8 & 0.13760 & 0.04303 & 0.02862 & 0.02129 & 0.01719 & 0.01477 & 0.01434 &  &  \\
  9 & 0.13742 & 0.04285 & 0.02840 & 0.02102 & 0.01683 & 0.01423 & 0.01262 & 0.01254 &  \\
  10 & 0.13728 & 0.04271 & 0.02824 & 0.02083 & 0.01660 & 0.01391 & 0.01214 & 0.01101 & 0.01114 \\
   \hline
\end{tabular}
\end{table}

\begin{table}
\caption{Coefficients $\Gamma_n^k$ for $H=0.7$}\label{tab:0.7}
\centering
\begin{tabular}{c|ccccccccc}
  \hline
$n\backslash k$ & 2 & 3 & 4 & 5 & 6 & 7 & 8 & 9 & 10\\
  \hline
2 & 0.31951 &  &  &  &  &  &  &  &  \\
  3 & 0.28867 & 0.09652 &  &  &  &  &  &  &  \\
  4 & 0.28207 & 0.07677 & 0.06840 &  &  &  &  &  &  \\
  5 & 0.27860 & 0.07288 & 0.05409 & 0.05074 &  &  &  &  &  \\
  6 & 0.27654 & 0.07069 & 0.05114 & 0.03946 & 0.04048 &  &  &  &  \\
  7 & 0.27518 & 0.06936 & 0.04942 & 0.03708 & 0.03117 & 0.03366 &  &  &  \\
  8 & 0.27421 & 0.06846 & 0.04835 & 0.03566 & 0.02917 & 0.02573 & 0.02881 &  &  \\
  9 & 0.27348 & 0.06782 & 0.04762 & 0.03476 & 0.02796 & 0.02401 & 0.02191 & 0.02518 &  \\
  10 & 0.27292 & 0.06733 & 0.04708 & 0.03413 & 0.02718 & 0.02295 & 0.02039 & 0.01907 & 0.02237 \\
   \hline
\end{tabular}
\end{table}

\begin{table}
\caption{Coefficients $\Gamma_n^k$ for $H=0.8$}\label{tab:0.8}
\centering
\begin{tabular}{c|ccccccccc}
  \hline
$n\backslash k$ & 2 & 3 & 4 & 5 & 6 & 7 & 8 & 9 & 10\\
  \hline
2 & 0.51572 &  &  &  &  &  &  &  &  \\
  3 & 0.44379 & 0.13947 &  &  &  &  &  &  &  \\
  4 & 0.42915 & 0.09287 & 0.10500 &  &  &  &  &  &  \\
  5 & 0.42108 & 0.08574 & 0.07202 & 0.07684 &  &  &  &  &  \\
  6 & 0.41637 & 0.08132 & 0.06676 & 0.05103 & 0.06130 &  &  &  &  \\
  7 & 0.41325 & 0.07873 & 0.06336 & 0.04690 & 0.04010 & 0.05089 &  &  &  \\
  8 & 0.41103 & 0.07698 & 0.06132 & 0.04414 & 0.03668 & 0.03291 & 0.04352 &  &  \\
  9 & 0.40938 & 0.07573 & 0.05993 & 0.04246 & 0.03435 & 0.02999 & 0.02790 & 0.03801 &  \\
  10 & 0.40810 & 0.07479 & 0.05892 & 0.04130 & 0.03292 & 0.02796 & 0.02534 & 0.02420 & 0.03373 \\
   \hline
\end{tabular}
\end{table}

\begin{table}
\caption{Coefficients $\Gamma_n^k$ for $H=0.9$}\label{tab:0.9}
\centering
\begin{tabular}{c|ccccccccc}
  \hline
$n\backslash k$ & 2 & 3 & 4 & 5 & 6 & 7 & 8 & 9 & 10\\
  \hline
2 & 0.74110 &  &  &  &  &  &  &  &  \\
  3 & 0.60809 & 0.17948 &  &  &  &  &  &  &  \\
  4 & 0.58213 & 0.09152 & 0.14465 &  &  &  &  &  &  \\
  5 & 0.56714 & 0.08204 & 0.08433 & 0.10362 &  &  &  &  &  \\
  6 & 0.55857 & 0.07506 & 0.07754 & 0.05671 & 0.08272 &  &  &  &  \\
  7 & 0.55290 & 0.07118 & 0.07223 & 0.05156 & 0.04445 & 0.06852 &  &  &  \\
  8 & 0.54889 & 0.06858 & 0.06921 & 0.04734 & 0.04028 & 0.03617 & 0.05851 &  &  \\
  9 & 0.54590 & 0.06673 & 0.06716 & 0.04492 & 0.03675 & 0.03266 & 0.03049 & 0.05105 &  \\
  10 & 0.54359 & 0.06535 & 0.06568 & 0.04326 & 0.03472 & 0.02962 & 0.02747 & 0.02633 & 0.04527 \\
   \hline
\end{tabular}
\end{table}

\begin{table}
\caption{Coefficients $\Gamma_n^k$ for $H=0.99$}\label{tab:0.99}
\centering
\begin{tabular}{c|ccccccccc}
  \hline
$n\backslash k$ & 2 & 3 & 4 & 5 & 6 & 7 & 8 & 9 & 10\\
  \hline
2 & 0.97247 &  &  &  &  &  &  &  &  \\
  3 & 0.76506 & 0.21328 &  &  &  &  &  &  &  \\
  4 & 0.72588 & 0.07275 & 0.18368 &  &  &  &  &  &  \\
  5 & 0.70233 & 0.06342 & 0.09059 & 0.12824 &  &  &  &  &  \\
  6 & 0.68917 & 0.05413 & 0.08408 & 0.05617 & 0.10262 &  &  &  &  \\
  7 & 0.68047 & 0.04937 & 0.07696 & 0.05159 & 0.04422 & 0.08474 &  &  &  \\
  8 & 0.67435 & 0.04617 & 0.07323 & 0.04602 & 0.04065 & 0.03555 & 0.07228 &  &  \\
  9 & 0.66979 & 0.04393 & 0.07067 & 0.04312 & 0.03604 & 0.03265 & 0.02980 & 0.06299 &  \\
  10 & 0.66628 & 0.04227 & 0.06885 & 0.04111 & 0.03363 & 0.02870 & 0.02735 & 0.02560 & 0.05582 \\
   \hline
\end{tabular}
\end{table}
\end{small}

Observe that
\begin{enumerate}
\item All coefficients are positive.

\item The first coefficient in each row is the largest, i.e.\ $\Gamma_n^2>\Gamma_n^k$ for any $3\le k \le n$. Moreover, often it is substantially larger than any other coefficient in the row.

\item The conjecture concerning the monotonicity of coefficients (decrease along each row) does not hold in general. If we take sufficiently large values of $H$, for example $H=0.9$, we see that the coefficient $\Gamma_n^3$ is always less than $\Gamma_4$. Moreover, the last coefficient $\Gamma_n^n$ is bigger than $\Gamma_n^{n-1}$ for sufficiently large~$H$.

\item The monotonicity along each column holds, i.e.\ $\Gamma_n^k > \Gamma_{n+1}^k$ for fixed $k$.

\item The limiting distribution of the coefficients as $H\uparrow1$ is not uniform.

\item It immediately follows from \eqref{eq:Cn+1k} that the coefficients satisfy the following relation:
\begin{gather*}
    \Gamma_{n+1}^{k+1}=\Gamma_{n}^{k+1}-\Gamma_{n+1}^{n+1}\Gamma_{n}^{n-k+1},
\end{gather*}
whence
\begin{gather*}
    \Gamma_{n+1}^{k+1}-\Gamma_{n+1}^{k}=\Gamma_{n}^{k+1}-\Gamma_{n}^{k}-\Gamma_{n+1}^{n+1}\left(\Gamma_{n}^{n-k+1}-\Gamma_{n}^{n-k+2}\right).
\end{gather*}
The second of these relations makes us expect that knowing that the coefficients decrease in $k$ for $n$ fixed and that the last coefficient $\Gamma_{n+1}^{n+1}$ is positive, we can prove that they decrease in $k$ for $n\rightsquigarrow n+1$ by induction.  Unfortunately, if we take $n=3$ as the start of the induction, we see that such a relation holds only if $k=2$, and indeed, $\Gamma_4^2> \Gamma_4^3$ as we know from Proposition  \ref{prop:n4-monot}. But the relation between $\Gamma_4^3$ and $\Gamma_4^4$ is not so stable and depends on $H$, see Figure \ref{fig:n=4}. Therefore, we cannot state that $\Gamma_4^3>\Gamma_4^4$.
\end{enumerate}

\subsection{Comparison of the methods. Computation time}
Let us compare the two methods in terms of computation time. The first method (solving the system of equations) was realized using the R function \texttt{solve()}. We considered two problems:
\begin{enumerate}
\item For fixed $n$, compute the coefficients $\Gamma_n^k$, $2\le k\le n$, i.e.\ compute the $n$th row of the matrix.
\item Compute the whole triangular array $\set{\Gamma_m^k\mid 2\le m\le n, 2\le k\le m}$. This requires solving $(n-1)$ systems of equations.
\end{enumerate}
The second method (recurrence relations) always gives us the whole array of coefficients, which can be considered as an advantage.

Let us mention that both methods give exactly the same values of the coefficients.

We also compared the time needed for computation on an Intel Core i3-8145U processor by each method. The results are shown in Table~\ref{tab:time}.
Observe that the recurrence method is always faster, especially for large $n$, if we need to compute the whole matrix. It takes less than $2$ seconds for $n=2000$, while solving all systems of equations takes more that $21$ minutes. Moreover, for large $n$ the recurrence method is even faster than the calculation of a single row of the matrix, which requires solving only one system of equations.

\begin{table}[h]
\caption{Computation time}\label{tab:time}
\begin{tabular}{lllll}
\toprule
$n$ & 100 & 500 & 1000 & 2000\\
\midrule
System method (last row) & 0.02 secs & 0.20 secs & 0.51 secs & 3.10 secs\\
System method (whole matrix) & 0.17 secs & 16.46 secs & 2.27 mins & 21.78 mins\\
Recurrence method & 0.04 secs & 0.19 secs & 0.48 secs & 1.83 secs\\
\bottomrule
\end{tabular}
\end{table}

\subsection{Remarks on the case $H<\frac12$}\label{ssec:H<1/2}

In this paper we mainly focus on the case $H>\frac12$ (the case of long-range dependence). In this section we give some brief comment on the other case $H<\frac12$.

\medskip
1. Using the complete monotonicity of $-\rho$ (see Lemma~\ref{l:cm}), we can show that in the case $H<\frac12$, the inequalities for $\rho_k$ from Corollary~\ref{c:rho-properties}, Properties 1 and 2, remain valid with opposite signs (the sign \enquote{$<$} instead of the sign \enquote{$>$}). In other words, the sequence $\set{\rho_k, k\ge0}$ is negative, increasing and concave. However, it remains log-convex, i.e.\ Property 3 of Corollary~\ref{c:rho-properties} holds for all $H\in (0,1/2)\cup(1/2,1)$.

\medskip
2. The behaviour of the coefficients for $n=3,4,5$ is shown in Figures \ref{fig:n=3b}--\ref{fig:n=5b}. We see that for $H<\frac12$, the coefficients are negative and increasing w.r.t.\ $H$. Moreover, for all $H<\frac12$ we also observe the monotonicity w.r.t.\ $k$, i.e. $\Gamma_n^k<\Gamma_n^{k+1}$ (unlike the case $H>\frac12$).

\begin{figure}
    \centering
    \begin{minipage}{0.47\textwidth}
        \centering
        \includegraphics[width=\textwidth]{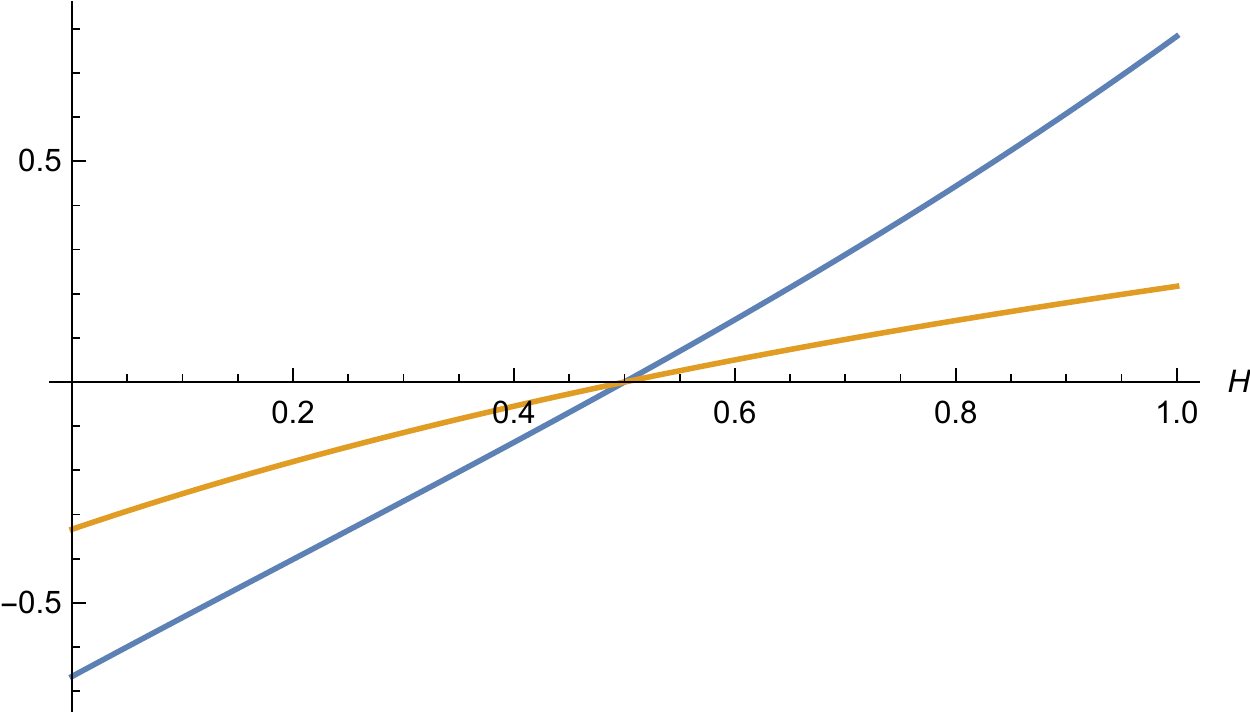} 
        \caption{Case $n=3$: $\Gamma_3^2$ and $\Gamma_3^3$ as the functions of $H$}\label{fig:n=3b}
    \end{minipage}\hfill
    \begin{minipage}{0.47\textwidth}
        \centering
        \includegraphics[width=\textwidth]{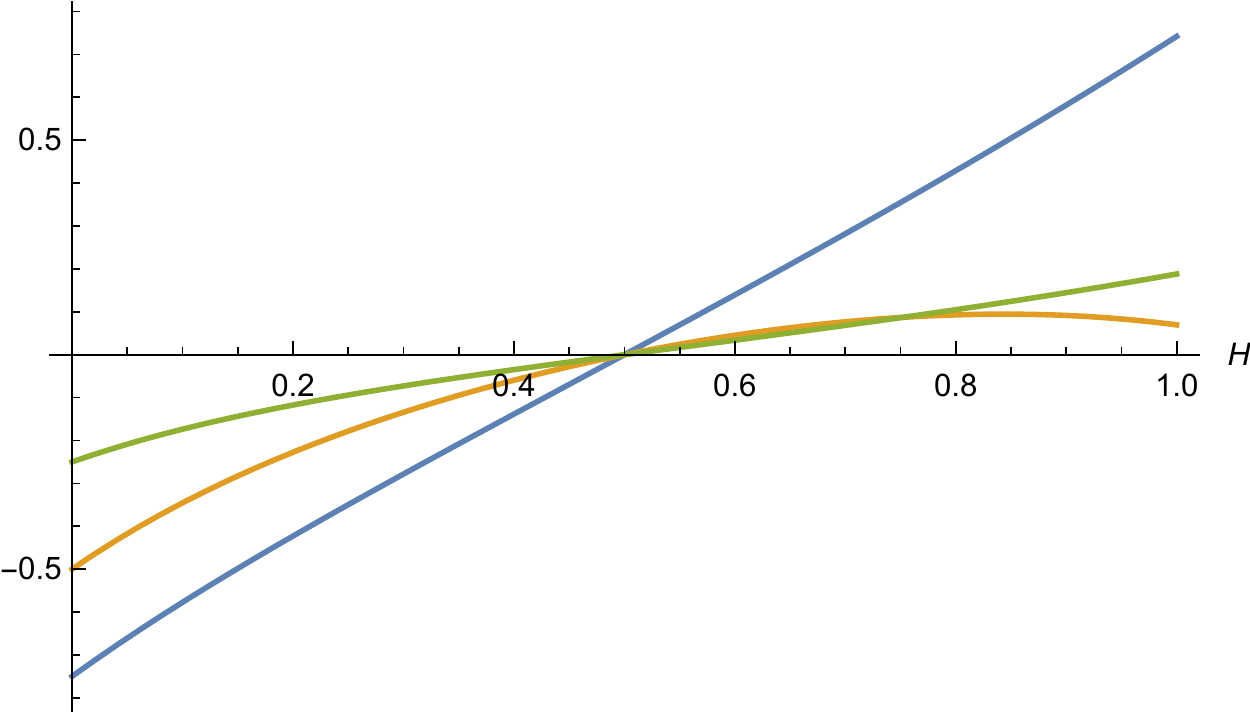} 
        \caption{Case $n=4$: $\Gamma_4^2$, $\Gamma_4^3$, and $\Gamma_4^4$ as the functions of $H$}\label{fig:n=4b}
    \end{minipage}
    \\[12pt]
    \begin{minipage}{0.47\textwidth}
        \centering
        \includegraphics[width=\textwidth]{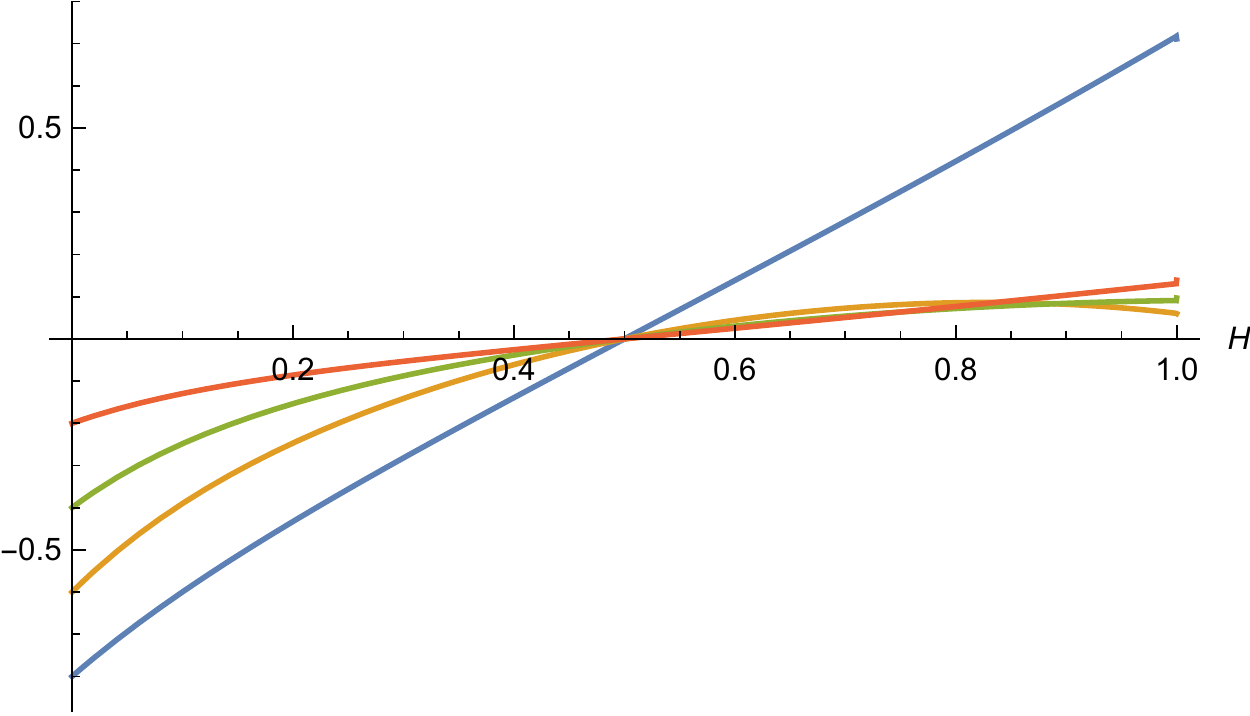} 
        \caption{Case $n=5$: $\Gamma_5^2$, $\Gamma_5^3$, $\Gamma_5^4$, and $\Gamma_5^5$ depending on $H$}\label{fig:n=5b}
    \end{minipage}
\end{figure}

\medskip
3. Let $H=0$.
In this case $B^H_n = B^0_n = \left(\xi_n-\xi_0\right)/{\sqrt{2}}$,
where $\{\xi_i,i\ge0\}$ is a sequence of i.i.d.\ $\Normal(0,1)$ random variables.
So, $\Delta_1 = \left({\xi_1-\xi_0}\right)/{\sqrt{2}}$ and, in general,
\begin{gather*}
    \Delta_k = \frac{\xi_k-\xi_{k-1}}{\sqrt{2}},\quad k\ge1.
\end{gather*}
Consider the equality
\begin{gather*}
    \Ee \left(\Delta_1 \mid \Delta_2,\dots,\Delta_n\right)
    = \sum_{k=2}^n \Gamma_n^k \Delta_k,
    \quad n\ge2.
\end{gather*}
Then
\begin{align*}
    \Ee \Delta_1 \Delta_2     &= -\frac12,          & \Ee \Delta_1 \Delta_k &= 0,\; k>2,\\
    \Ee \Delta_k \Delta_{k+1} &= -\frac12,\; k\ge2, & \Ee \Delta_k \Delta_l &= 0,\; |l-k|>1.
\end{align*}
Therefore, the system of linear equations has the form
\begin{gather*}
    \left\{\begin{aligned}
    -\frac12    &= \Gamma_n^2 - \frac12 \Gamma_n^3,\\
    0           &= -\frac12 \Gamma_n^k + \Gamma_n^{k+1} - \frac12 \Gamma_n^{k+2}, & 2\le k\le n-2,\\
    0           &= -\frac12 \Gamma_n^{n-1} + \Gamma_n^n,
    \end{aligned}\right.
\end{gather*}
and we get
\begin{gather*}
    \Gamma_n^n = -\tfrac12 \Gamma_n^{n-1},\quad
    \Gamma_n^2 =  \tfrac12 \Gamma_n^3 - \tfrac12,\quad
    \Gamma_n^3 =  \tfrac23 \Gamma_n^4 - \tfrac13,\dots,\\
    \Gamma_n^k =  \tfrac{k-1}{k} \Gamma_n^{k+1} - \tfrac1k,
    \dots,
    \Gamma_n^{n-1} =  \tfrac{n-2}{n-1} \Gamma_n^{n} - \tfrac1{n-1}.
\end{gather*}
Finding $\Gamma_n^n$ and $\Gamma_n^{n-1}$ from the first and last equations, and then calculating successively $\Gamma_n^{n-2},\dots,\Gamma_n^2$, we get the following solution
\begin{gather*}
\Gamma_n^k = \frac{n-k+1}{n},\quad k=2,\dots,n.
\end{gather*}

\end{document}